\documentclass[pdflatex,sn-mathphys-num]{sn-jnl}

\usepackage{graphicx,multirow,amsmath,amssymb,amsfonts,mathtools,bm}
\usepackage{amsthm,booktabs,array,float,placeins,enumitem}
\usepackage{algorithm,algorithmicx,algpseudocode}
\graphicspath{{figures/}}
\setlist{nosep}
\theoremstyle{thmstyleone}
\newtheorem{theorem}{Theorem}[section]
\newtheorem{proposition}[theorem]{Proposition}
\theoremstyle{thmstyletwo}
\newtheorem{remark}[theorem]{Remark}
\newtheorem{assumption}[theorem]{Assumption}

\newcommand{\C}{\mathbb C}
\newcommand{\R}{\mathbb R}
\newcommand{\Th}{\mathcal T_h}

\newcommand{\Ei}{\mathcal E_h^{\rm I}}

\newcommand{\jump}[1]{\left[\!\left[#1\right]\!\right]}
\newcommand{\jumpn}[1]{\left[\!\left[#1\right]\!\right]_{n}}
\newcommand{\sjump}[1]{[#1]}
\newcommand{\avg}[1]{\left\{\!\left\{#1\right\}\!\right\}}

\newcommand{\GR}{\mathrm{GR}}

\hypersetup{
  pdftitle={Direction-Adaptive Plane-Wave Discontinuous Galerkin Methods for the Helmholtz Equation},
  pdfauthor={Shelvean Kapita},
  pdfsubject={Plane-wave discontinuous Galerkin methods for the Helmholtz equation},
  pdfkeywords={Helmholtz equation, PWDG, Trefftz method, direction adaptivity, variable projection}
}
\begin{document}
\title[Direction-Adaptive PWDG for Helmholtz]{Direction-Adaptive Plane-Wave Discontinuous Galerkin Methods for the Helmholtz Equation}
\author*[1]{\fnm{Shelvean} \sur{Kapita}}
\affil*[1]{\orgdiv{Department of Mathematics}, \orgname{Texas A\&M University}, \orgaddress{\city{College Station}, \state{Texas}, \country{USA}}}
\abstract{We consider plane-wave discontinuous Galerkin (PWDG) approximations of the Helmholtz equation in which the local propagation directions are allowed to change.  The directions are determined by minimizing a weighted residual on the mesh skeleton.  We consider two formulations.  In Part~A the PWDG equations are retained and, for each set of directions, the plane-wave coefficients are obtained from the Galerkin system.  In Part~B the same skeleton residual is minimized over both coefficients and directions, with the coefficient variables eliminated by variable projection.  Complex angles are used so that propagating and evanescent waves are included in the same local Trefftz family.  We also use the Trefftz-DG norm to normalize the discrete system and a local Cauchy-trace Gramian to remove numerically dependent directions.  On straight edges the plane-wave products in the residual can be integrated exactly.  For an all-Dirichlet problem the residual is the squared DG error, and its local contributions can also be used for adaptivity.  We prove a local quadratic-growth result for the reduced direction problem in a fixed-rank neighborhood of an identifiable zero-residual solution.  The numerical experiments first show a limitation: on an exact circular DtN boundary the dominant phase direction of a Hankel wave is recovered accurately, but a small plane-wave fan does not give a comparably small field error.  We then show that the high-$p$ DtN floor is affected separately by the coefficient basis, the trace cutoff and the arithmetic used to compute the trace spectrum.  Finally, for a finite sum of plane waves, residual-based ENRICH--MOVE continuation recovers all directions to roundoff for $M=1,\ldots,19$ in the family considered.  At $M=20$ the automatic birth step enters a false basin, while a nearby birth at the same dimension again gives roundoff error.  These experiments show that direction adaptation is most effective when the solution has low directional complexity.}
\maketitle

\section{Introduction}\label{sec:intro}
The numerical solution of the Helmholtz equation becomes increasingly expensive as the wavenumber is increased.  One approach is to use oscillatory basis functions that already contain part of the behavior of the solution.  If the local basis functions satisfy the Helmholtz equation exactly, the resulting method is a Trefftz method.  Plane waves give a particularly simple Trefftz space and lead to numerical formulations involving only traces on the mesh skeleton.  The ultra-weak variational formulation of Cessenat and Despr\'es \cite{CessenatDespres1998} is an early example.  Plane-wave discontinuous Galerkin methods were subsequently analyzed in \cite{GittelsonHiptmairPerugia2009,HiptmairMoiolaPerugia2011,HiptmairMoiolaPerugia2016}; see also the survey \cite{HiptmairMoiolaPerugiaSurvey2016}.  Residual based mesh adaptivity for PWDG was considered in \cite{KapitaMonkWarburton2015}.

In the usual PWDG method, the number of plane waves and their directions are chosen when the discrete space is constructed.  Uniformly spaced directions are convenient and work well in many examples.  They can, however, be wasteful when the solution is dominated by only a few propagation directions.  Adding more uniformly spaced waves is also not without cost.  As the local dimension is increased, the traces of the plane waves become nearly dependent and the resulting algebraic problem can be severely ill-conditioned \cite{CongreveGedickePerugia2019}.  This suggests treating the directions as part of the approximation problem rather than fixing them beforehand.

This idea has been considered in several forms.  Amara et al. \cite{AmaraChaudhryDiazDjellouliFiedler2014} rotate local plane-wave bases in a least-squares method and determine the rotations from a nonlinear minimization problem.  Agrawal and Hoppe \cite{AgrawalHoppe2017} optimize real PWDG directions while keeping the PWDG variational equation as a constraint.  Congreve, Houston and Perugia \cite{CongreveHoustonPerugia2019} use directional information together with local enrichment and mesh refinement in an $hp$-Trefftz DG method.  Fang et al. \cite{FangQianZepedaZhao2017} first compute a lower-frequency solution, extract dominant ray directions, and use those directions in a higher-frequency finite element space.  Other methods using approximate propagation information are described in \cite{LamQian2019,HuWang2021}.

Our aim is to study a PWDG calculation in which the directions are determined from the same residual that is later used to judge the quality of the approximation.  The residual contains the jumps of the field and normal flux and the mismatch in the boundary conditions.  We allow the angle of a plane wave to be complex.  Hence an ordinary propagating wave and an evanescent wave are obtained from the same formula.  No separate evanescent basis or critical-angle test is required.

We use two related formulations.  In the first, which we call Part~A, the PWDG equations are kept.  For a given set of directions we solve the PWDG problem, substitute the computed field into the skeleton residual, and then vary only the directions.  Thus the coefficients satisfy the Galerkin equations at every nonlinear iterate.  This is closest to the formulation of Agrawal and Hoppe \cite{AgrawalHoppe2017}.  In Part~B we remove the Galerkin constraint and minimize the skeleton residual over the broken Trefftz field itself.  With the directions fixed, this is a linear least-squares problem of the type studied by Monk and Wang \cite{MonkWang1999}.  We then eliminate the linear coefficient variables by variable projection \cite{GolubPereyra1973}.

There are two algebraic issues that arise immediately in such a calculation.  The first is conditioning.  The Euclidean condition number of the PWDG matrix depends strongly on the scaling of the plane-wave coefficients and does not by itself describe the discrete operator in its natural norm.  We therefore use the standard coercive Trefftz-DG norm as a Riesz metric.  This gives a graph--Riesz condition number for the normalized operator.  The second issue is numerical rank.  If several local directions are close, some combinations of their Cauchy traces contain essentially no new information.  We detect these combinations with a local trace Gram matrix and remove them before the global solve.  Related conditioning and redundancy issues for Trefftz spaces are discussed in \cite{BarucqBendaliDiazTordeux2021,BetckeTrefethen2005,BarnettBetcke2008,ParolinHuybrechsMoiola2023}.

For straight-sided elements, the edge integrals needed by the residual can be evaluated directly.  Products of two plane waves reduce to exponential moments, and the same is true after differentiation with respect to a direction parameter.  Therefore the least-squares matrix can be assembled without sampling the homogeneous edge residuals.  There is a further simplification in the all-Dirichlet case.  If $u$ is the exact solution and $v_h$ is any broken Trefftz field, then
\[
 \mathcal J(v_h)=\mathcal G_h(u-v_h,u-v_h)=\|u-v_h\|_{\mathcal G,h}^2.
\]
Thus the quantity minimized by the direction search is also the exact squared DG/graph error.  Its edge contributions give local error quantities without introducing a second estimator.  For mixed Dirichlet--impedance conditions this exact identity is lost.  In that case we obtain an $L^2$ bound from a Helmholtz stability estimate stated later in the paper.

The numerical experiments are arranged to show first where direction movement helps and then where it does not.  We first use three hidden plane waves to test exact recovery, and a transmission problem to check that the same complex-angle parameterization recovers propagating and evanescent waves.  We next return to the exact circular DtN geometry of \cite{KapitaMonk2018}.  For a Hankel field the learned directions follow the local phase very well, but the field error remains limited by curvature and amplitude variation.  We then increase $p$ on the same geometry and investigate separately the effects of the raw coefficient basis, the trace cutoff and arithmetic precision.  Our main recovery experiment is considered after these tests.  The exact field is a finite sum of $M$ plane waves, and $M$ is increased one direction at a time.  For the family used here the residual-based ENRICH--MOVE iteration gives roundoff error for $M=1,\ldots,19$.  At $M=20$ the automatic birth step enters a false basin, while a nearby birth at the same dimension again converges to roundoff.  Thus the important distinction in the experiments is between a field with a small number of identifiable directions and a field with general angular content.  For the more nonconvex examples we use continuation, initializing a more difficult problem with directions obtained from the preceding one \cite{Watson1989}.

The paper is organized as follows.  In Section~\ref{sec:pwdg} we give the PWDG formulation and introduce the complex-angle Trefftz space.  Section~\ref{sec:graph} describes the graph--Riesz normalization and local trace-rank reduction.  In Section~\ref{sec:nonlinear} we define the skeleton residual and the two nonlinear formulations, and we derive the local fixed-rank result and the direction derivatives.  The numerical algorithms are summarized in Section~\ref{sec:algorithm}.  Section~\ref{sec:numerics} contains the numerical experiments.  We end with conclusions in Section~\ref{sec:conclusion}.

\section{PWDG with Variable Complex Directions}\label{sec:pwdg}

\subsection{Model problem}
We begin with a standard truncated acoustic scattering problem.  Let $\Omega\subset\mathbb R^2$ be a bounded Lipschitz domain with boundary
\[
 \partial\Omega=\Gamma_D\,\dot\cup\,\Gamma_R.
\]
Here $\Gamma_D$ denotes the boundary on which Dirichlet data are prescribed and $\Gamma_R$ is an artificial exterior boundary, as in the usual truncated scattering problem \cite{CessenatDespres1998,HiptmairMoiolaPerugiaSurvey2016}.  If $n$ is the outward unit normal to $\Omega$, we consider
\begin{equation}\label{eq:model}
 \begin{aligned}
  &\Delta u+\kappa^2u=0 \qquad \text{in }\Omega,\\
  &u=g \qquad \text{on }\Gamma_D,\\
  &\partial_nu+i\kappa u=0 \qquad \text{on }\Gamma_R,
 \end{aligned}
\end{equation}
where $\kappa>0$ is the wavenumber.  We use the time dependence $e^{+i\omega t}$.  With this convention an outgoing radial wave behaves like $e^{-i\kappa r}$, so the first-order absorbing condition is $\partial_nu+i\kappa u=0$.  Under the alternative convention $e^{-i\omega t}$ the outgoing factor is $e^{+i\kappa r}$ and the corresponding impedance sign is reversed.  The PWDG fluxes below use the $e^{+i\omega t}$ convention consistently.

Two special cases will be used in the numerical experiments.  In the first, $\Gamma_R=\varnothing$ and Dirichlet data are prescribed on the whole boundary of a square.  In the second, the wavenumber is piecewise constant and the field and normal flux are continuous across a material interface \cite{MoiolaSpence2019}.  The PWDG formulation below is written first for \eqref{eq:model}; the changes needed for these two examples are stated when they are used.

For an interior edge $e=\partial K^+\cap\partial K^-$ with outward normals $n^\pm$, we use the standard two-sided DG averages and jumps \cite{KapitaMonkWarburton2015,HiptmairMoiolaPerugia2011,HiptmairMoiolaPerugiaSurvey2016}.  For a scalar trace $v$,
\begin{equation}\label{eq:scalar-jump}
 \avg{v}=\frac12(v^++v^-),
 \qquad
 \jump{v}=v^+n^+ + v^-n^- ,
\end{equation}
so that $\jump{v}$ is vector-valued.  For a vector trace $\boldsymbol\tau$,
\begin{equation}\label{eq:vector-jump}
 \avg{\boldsymbol\tau}=\frac12(\boldsymbol\tau^++\boldsymbol\tau^-),
 \qquad
 \jumpn{\boldsymbol\tau}=\boldsymbol\tau^+\!\cdot n^+ + \boldsymbol\tau^-\!\cdot n^- ,
\end{equation}
so that $\jumpn{\boldsymbol\tau}$ is scalar-valued.  We reserve the double-line symbols $\jump{\cdot}$ and $\jumpn{\cdot}$ for these DG jumps throughout the paper.  Ordinary single brackets $\sjump{\cdot}$ are used later for an oriented scalar difference on a fixed edge, where $\sjump{v}=v^+-v^-$ after choosing one normal $n_e$.  They are not DG jumps.  The Monk--Wang comparison uses this oriented scalar difference for the full-gradient jump.  The normals are real, hence conjugation commutes with the trace operators.

\paragraph{Edge-scale convention.}
For the constant-wavenumber problem we take $\xi_e=\kappa$.  The solution jump is then weighted by $\alpha\kappa$ and the normal-flux jump by $\beta/\kappa$, as in the usual Trefftz-DG scaling \cite{HiptmairMoiolaPerugia2011,HiptmairMoiolaPerugiaSurvey2016}.  In the transmission experiment the two neighboring elements have different wavenumbers, and we use the arithmetic interface scale
\begin{equation}\label{eq:xi-interface}
 \xi_e:=\frac{\kappa_{K^+}+\kappa_{K^-}}{2}
 \qquad (e=\partial K^+\cap\partial K^-),
\end{equation}
and retain the reciprocal pair $\alpha\xi_e$ and $\beta/\xi_e$.  This choice of average is made only for the numerical transmission problem.  We do not use it as a general variable-coefficient flux prescription.  The remaining flux parameters satisfy $\alpha,\beta>0$ and $0<\delta\le1/2$; admissible choices and their role in the Trefftz-DG norm are discussed in \cite{GittelsonHiptmairPerugia2009,HiptmairMoiolaPerugia2011}.  We take $\alpha=\beta=1/2$ in the all-Dirichlet computations.  The parameter $\delta$ is needed only on $\Gamma_R$ and determines the split of the Robin boundary terms.

\subsection{Complex-angle Trefftz basis functions}
We next describe the local Trefftz space.  Let $x_K$ be a reference point in the element $K$; in the computations we take $x_K$ to be the centroid.  For a complex angle $z=\theta+i\eta\in\C$, define
\begin{equation}\label{eq:qz}
 q_K(z)=\kappa_K(\cos z,\sin z)
\end{equation}
and
\begin{equation}\label{eq:phiz}
 \phi_{K,z}(x)=\exp\big(iq_K(z)\cdot(x-x_K)\big).
\end{equation}
The dot product in \eqref{eq:phiz} is bilinear, without complex conjugation.

\begin{proposition}[complete complex-angle Trefftz family]\label{prop:trefftz}
For every $z\in\C$,
\[
 \Delta\phi_{K,z}+\kappa_K^2\phi_{K,z}=0\qquad\text{in }K.
\]
In particular, $\eta=0$ gives an ordinary propagating plane wave, while $\eta\ne0$ gives an inhomogeneous or evanescent Trefftz wave \cite{ParolinHuybrechsMoiola2023}.  Conversely, every $q\in\C^2$ satisfying the complex Helmholtz dispersion relation $q\cdot q=\kappa_K^2$ can be written as $q=q_K(z)$ for some $z\in\C$.
\end{proposition}
\begin{proof}
The complex trigonometric identity $\cos^2z+\sin^2z=1$ gives $q_K(z)\cdot q_K(z)=\kappa_K^2$.  Since $\nabla\phi_{K,z}=iq_K(z)\phi_{K,z}$,
\[
 \Delta\phi_{K,z}=-(q_K(z)\cdot q_K(z))\phi_{K,z}=-\kappa_K^2\phi_{K,z}.
\]
For the converse, let $q=(q_1,q_2)\in\C^2$ satisfy $q_1^2+q_2^2=\kappa_K^2$ and set $w=(q_1+iq_2)/\kappa_K$.  Then $w\ne0$ and $w^{-1}=(q_1-iq_2)/\kappa_K$.  Choose any logarithm branch at the nonzero number $w$ and take $z=-i\log w$.  Since $e^{iz}=w$,
\[
 \cos z=\frac{w+w^{-1}}2=\frac{q_1}{\kappa_K},\qquad
 \sin z=\frac{w-w^{-1}}{2i}=\frac{q_2}{\kappa_K}.
\]
Thus $q=q_K(z)$.
\end{proof}

Given local direction parameters
\[
 Z_K=(z_{K,1},\ldots,z_{K,r_K}),
\]
define
\begin{equation}\label{eq:VKZ}
 V_K(Z_K)=\operatorname{span}\{\phi_{K,z_{K,1}},\ldots,\phi_{K,z_{K,r_K}}\},
 \qquad
 V_h(Z)=\prod_{K\in\Th}V_K(Z_K).
\end{equation}
In a standard $p$-PWDG calculation the local directions are prescribed when the approximation space is defined \cite{HiptmairMoiolaPerugia2011,HiptmairMoiolaPerugia2016}.  In the present method the collection $Z$ will be varied by the nonlinear algorithm.

For numerical purposes the complex angles are restricted to a bounded set.  Write
$z=\theta+i\eta$ and fix $\eta_{\max}>0$.  We define
\begin{equation}\label{eq:ZK-adm}
 \mathcal Z_K(\eta_{\max})
 :=\big(\R/(2\pi\mathbb Z)\big)\times[-\eta_{\max},\eta_{\max}],
\end{equation}
and, for the current local multiplicities $\bm r=(r_K)_{K\in\Th}$,
\begin{equation}\label{eq:Zh-adm}
 \mathcal Z_h(\bm r,\eta_{\max})
 :=\prod_{K\in\Th}\mathcal Z_K(\eta_{\max})^{r_K}.
\end{equation}
All reported complex-angle experiments use $\eta_{\max}=1.8$.  This is the actual bound imposed in the reference code.

The imaginary component has the useful physical interpretation
\begin{equation}\label{eq:imq}
 \operatorname{Im}q_K(\theta+i\eta)
 =\kappa_K\sinh\eta\,(-\sin\theta,\cos\theta),
 \qquad
 \|\operatorname{Im}q_K(\theta+i\eta)\|_2
 =\kappa_K|\sinh\eta|.
\end{equation}
It follows that $h_K\kappa_K|\sinh\eta|$ is a useful measure of the exponential variation of the basis function over $K$.  One could impose a bound on this quantity instead of a fixed bound on $|\eta|$, but we do not do so in the computations below.  We also note that a permutation of the $r_K$ angles does not change $V_K(Z_K)$.  The nonlinear objective therefore has the corresponding permutation symmetry, although the implementation simply keeps the directions as labeled variables.

\subsection{Skeleton formulation}
We now write the PWDG formulation for fixed directions $Z$.  We use the convention that $a_h(\cdot,\cdot)$ is linear in the first argument and conjugate-linear in the second.  The mixed Dirichlet--impedance form is
\begin{align}\label{eq:pwdgform}
 a_h(u,v)=&\sum_{e\in\Ei}\int_e
 \Big(
 \avg{u}\,\jumpn{\nabla_h\overline v}
 -\jump{\overline v}\cdot\avg{\nabla_hu}
 +i\frac{\beta}{\xi_e}\jumpn{\nabla_hu}\,\jumpn{\nabla_h\overline v}
 \Big)\,ds \nonumber\\
 &\quad+\sum_{e\in\Ei}\int_e i\alpha\xi_e\jump{u}\cdot\jump{\overline v}\,ds \\
 &+\int_{\Gamma_D}
 \Big(-\partial_nu\,\overline v+i\alpha\kappa u\,\overline v\Big)\,ds \nonumber\\
 &+\int_{\Gamma_R}\Big(
 (1-\delta)i\kappa u\,\overline v
 +(1-\delta)u\,\partial_n\overline v
 -\delta\,\partial_nu\,\overline v
 -\delta(i\kappa)^{-1}\partial_nu\,\partial_n\overline v
 \Big)\,ds.\nonumber
\end{align}
The corresponding right-hand side for the homogeneous impedance condition in \eqref{eq:model} is
\begin{equation}\label{eq:rhs}
 F_h(v)=\int_{\Gamma_D}
 g\Big(i\alpha\kappa\overline v-\partial_n\overline v\Big)\,ds.
\end{equation}
Equations \eqref{eq:pwdgform}--\eqref{eq:rhs} are the standard primal Trefftz-DG/PWDG flux formulation for a Dirichlet--Robin problem \cite{HiptmairMoiolaPerugia2011,HiptmairMoiolaPerugiaSurvey2016}.  When $\Gamma_R=\varnothing$ the Robin terms disappear, and this is the form used for the all-Dirichlet numerical examples.

For a complex wave vector $q$, if $v(x)=e^{iq\cdot(x-x_K)}$, then the conjugated test traces are
\[
 \nabla_h\overline v=-i\,\overline q\,\overline v,
 \qquad
 \partial_n\overline v=-i(\overline q\cdot n)\overline v.
\]
The fixed-direction problem is
\begin{equation}\label{eq:fixedPWDG}
 \text{find }u_h(Z)\in V_h(Z):\qquad
 a_h(u_h(Z),v_h)=F_h(v_h)\quad\forall v_h\in V_h(Z).
\end{equation}
Since every function in $V_h(Z)$ satisfies the Helmholtz equation on each element, the volume terms cancel after elementwise integration by parts and the resulting discrete problem involves only the mesh skeleton \cite{CessenatDespres1998,HiptmairMoiolaPerugiaSurvey2016}.

Fix an ordering of all active local Trefftz basis functions and write the resulting global broken Trefftz basis as
\[
 \Phi(Z)=\{\varphi_1(Z),\ldots,\varphi_{N_h}(Z)\}.
\]
With
\[
 u_h(Z)=\sum_{j=1}^{N_h}c_j(Z)\varphi_j(Z),
\]
the assembled matrix and load vector are
\begin{equation}\label{eq:matrix-def}
 A_{ij}(Z)=a_h\big(\varphi_j(Z),\varphi_i(Z)\big),
 \qquad
 f_i(Z)=F_h\big(\varphi_i(Z)\big).
\end{equation}
It follows that the fixed-direction PWDG problem can be written as
\begin{equation}\label{eq:matrix-system}
 A(Z)c(Z)=f(Z).
\end{equation}
For later comparison, we define the ordinary Euclidean condition number by
\begin{equation}\label{eq:kappaA}
 \kappa_2(A(Z))=\|A(Z)\|_2\,\|A(Z)^{-1}\|_2
\end{equation}
in the coefficient basis used to assemble $A$.  This number changes with the choice and scaling of the basis.  In the next section we introduce a second condition number after normalization by the DG/graph norm.

\section{Graph Norm, Normalization and Numerical Rank}\label{sec:graph}

\subsection{The DG/graph norm and graph--Riesz normalization}
We first introduce the norm that will be used to normalize the discrete operator.  With the same linear-first convention as for $a_h$, define
\begin{equation}\label{eq:graphform}
\begin{aligned}
 \mathcal G_h(u,v):={}&
 \sum_{e\in\Ei}\left[
 \alpha\xi_e\int_e \jump u\cdot\overline{\jump v}\,ds
 +\frac{\beta}{\xi_e}\int_e\jumpn{\nabla_hu}\,\overline{\jumpn{\nabla_hv}}\,ds
 \right]\\
 &+\alpha\kappa\int_{\Gamma_D}u\overline v\,ds
 +\int_{\Gamma_R}\left[(1-\delta)\kappa u\overline v
 +\frac{\delta}{\kappa}\partial_nu\,\overline{\partial_nv}\right]ds .
\end{aligned}
\end{equation}
For Trefftz functions, the standard PWDG Green identity gives \cite{GittelsonHiptmairPerugia2009,HiptmairMoiolaPerugia2011,HiptmairMoiolaPerugiaSurvey2016}
\begin{equation}\label{eq:coerciveidentity}
 \operatorname{Im}a_h(v,v)=\mathcal G_h(v,v).
\end{equation}
If the retained Trefftz traces are independent and the fixed-direction PWDG problem is uniquely solvable, $\mathcal G_h(v,v)=0$ implies $v=0$.  Hence $\|v\|_{\mathcal G,h}=\mathcal G_h(v,v)^{1/2}$ is a norm on the retained space and its Gram matrix is Hermitian positive definite.  When two or more local traces become nearly dependent, the smallest eigenvalues of this Gram matrix become very small.  We therefore remove numerically dependent local trace combinations before applying the Cholesky normalization described below.

\begin{remark}[Relation with the standard DG norm]\label{rem:graph-DG}
For constant $\kappa$, this is exactly the usual Trefftz-DG/PWDG coercive norm \cite{HiptmairMoiolaPerugia2011,HiptmairMoiolaPerugiaSurvey2016}:
\begin{align}\label{eq:DGnorm-identification}
 |||v|||_{\rm DG}^2
 ={}&\frac{\beta}{\kappa}\|\jumpn{\nabla_hv}\|_{L^2(\Ei)}^2
 +\alpha\kappa\|\jump v\|_{L^2(\Ei)}^2
 +\alpha\kappa\|v\|_{L^2(\Gamma_D)}^2\\
 &+\frac{\delta}{\kappa}\|\partial_n v\|_{L^2(\Gamma_R)}^2
 +(1-\delta)\kappa\|v\|_{L^2(\Gamma_R)}^2
 =\|v\|_{\mathcal G,h}^2.\nonumber
\end{align}
We use the name \emph{graph norm} because the Gram matrix of this norm will be used as the coefficient-space Riesz map.  This should not be confused with the stronger $DG+$ norm that appears in continuity estimates.  In the transmission tests, $\kappa^{\pm1}$ on interior edges is replaced by $\xi_e^{\pm1}$ from \eqref{eq:xi-interface}.  Thus the heterogeneous implementation changes only the positive edge scale.  It preserves the reciprocal solution-jump/flux-jump balance used in the coercive identity.
\end{remark}

\paragraph{Weight conventions.}
The PWDG operator and graph metric use $\alpha\xi_e$ and $\beta/\xi_e$ on the interior solution and flux jumps.  The residual $\mathcal J$ uses the same interior and Dirichlet weights but the strong Robin residual with factor $(2\kappa)^{-1}$.  The local rank metric $m_K$ instead uses only the dimensional pair $\kappa_K,\kappa_K^{-1}$.  It is a rank diagnostic and does not enter the Galerkin form or the residual objective.

Let $A(Z)$ be the PWDG matrix in the active basis.  The Gram matrix of \eqref{eq:graphform} is
\begin{equation}\label{eq:Gmatrix}
 G_{ij}(Z)=\mathcal G_h(\varphi_j,\varphi_i)
 =\left[\frac{A-A^*}{2i}\right]_{ij},
 \qquad G=\frac{A-A^*}{2i}.
\end{equation}
If $G>0$, choose $B$ so that
\begin{equation}\label{eq:Bdef}
 B^*GB=I.
\end{equation}
In the computations $G=LL^*$ and $B$ is obtained from the triangular solve $L^*B=I$.  Writing $A=H+iG$ with $H=H^*$ gives the exact normal form
\begin{equation}\label{eq:normalform}
 \widehat A:=B^*AB=S+iI,
 \qquad S=B^*HB=S^*.
\end{equation}
It follows that $\widehat A$ is normal and $\sigma_{\min}(\widehat A)\ge1$.  With $c=B\widehat c$,
\begin{equation}\label{eq:gr-solve}
 (B^*AB)\widehat c=B^*f,
 \qquad c=B\widehat c,
\end{equation}
so the Galerkin field is unchanged.  The dense graph--Riesz realization is used here as a reference normalization.  Large-scale variants can approximate the Riesz map locally or by Schwarz/multilevel techniques \cite{MardalWinther2011}.

We measure intrinsic conditioning by
\begin{equation}\label{eq:kgr}
 \kappa_{\GR}=\kappa_2(B^*AB).
\end{equation}
If $\lambda_j$ are the eigenvalues of $S$, normality gives
\begin{equation}\label{eq:kgr-exact}
 {\kappa_{\GR}
 =\left(\frac{1+\max_j\lambda_j^2}{1+\min_j\lambda_j^2}\right)^{1/2}
 \le \sqrt{1+\gamma_h^2}},
 \qquad \gamma_h=\|S\|_2.
\end{equation}
The quantity $\kappa_{\GR}$ therefore measures the size of the Hermitian part of the discrete operator relative to the coercive skeleton metric.  In particular, it is not simply a measure of how well the original plane-wave coefficients have been scaled.

Table~\ref{tab:uniform-conditioning} shows the difference between the two condition numbers for the uniform 32-triangle transmission mesh at $\omega=12$.  As $p$ is increased, the Euclidean condition number grows very rapidly.  The graph--Riesz condition number grows much more slowly until the local traces themselves lose numerical rank.  Once this happens, normalization alone is not sufficient and a rank reduction is needed.
\begin{table}[htbp]
\centering
\caption{Ordinary and graph--Riesz conditioning for uniformly directed PWDG on the uniform 32-triangle mesh at $\omega=12$.}
\label{tab:uniform-conditioning}
\begin{tabular}{|r|r|r|r|}
\hline
$p$ & coefficients & $\kappa_2(A)$ & $\kappa_{\GR}$\\
\hline
3  &  96 & $2.87\times10^{1}$  & $2.70$\\
7  & 224 & $5.90\times10^{1}$  & $4.01$\\
11 & 352 & $3.21\times10^{2}$  & $1.53\times10^{1}$\\
15 & 480 & $4.47\times10^{4}$  & $1.99\times10^{1}$\\
19 & 608 & $3.37\times10^{7}$  & $4.46\times10^{1}$\\
23 & 736 & $7.76\times10^{10}$ & $8.12\times10^{2}$\\
27 & 864 & $3.43\times10^{14}$ & $1.69\times10^{3}$\\
29 & 928 & $2.48\times10^{17}$ & ---\\
\hline
\end{tabular}
\end{table}
At $p=29$ the smallest computed graph eigenvalue is $-2.0\times10^{-14}$, a roundoff signature of numerical rank loss.  We therefore report $\kappa_{\GR}$ as undefined before trace compression.

\subsection{Local trace-rank compression}\label{subsec:compression}
We now consider what happens when several plane waves are present on the same element.  Near linear dependence of high-order plane-wave traces is well known \cite{CongreveGedickePerugia2019,BarucqBendaliDiazTordeux2021,ParolinHuybrechsMoiola2023}.  Rather than testing rank with the assembled global graph matrix, we use a local Cauchy-trace Gram matrix.  For $K\in\mathcal T_h$, define
\begin{equation}\label{eq:localmetric}
 m_K(u,v)=\sum_{e\subset\partial K}\left(
 \kappa_K\int_eu\overline v\,ds
 +\kappa_K^{-1}\int_e\partial_nu\,\overline{\partial_nv}\,ds\right).
\end{equation}
The factors $\kappa_K$ and $\kappa_K^{-1}$ balance the two traces in the same dimensions as the PWDG graph weights.  For the homogeneous calculations with $\alpha=\beta=1/2$ the two scalings are proportional before the element traces are assembled into global edge terms.  We stress, however, that $m_K$ is used only to determine numerical rank.  It is independent of $\alpha,\beta,\delta$ and $\xi_e$, and it is not the restriction of $\mathcal G_h$ to one element.  For the local basis $\Phi_K=(\phi_{K,z_1},\ldots,\phi_{K,z_{r_K}})$ let
$M_K=(m_K(\phi_{K,z_j},\phi_{K,z_i}))_{ij}=Q_K\Lambda_KQ_K^*$, with
$\lambda_{K,1}\ge\cdots\ge0$.  The matrix $M_K$ is the Gram matrix of the local Cauchy traces in the metric $m_K$.  Hence $M_K$ is Hermitian positive semidefinite and, for a coefficient vector $a\in\mathbb C^{r_K}$,
\[
 m_K(\Phi_Ka,\Phi_Ka)=a^*M_Ka.
\]
Thus an eigenvector $q_{K,j}$ gives a particular linear combination $\Phi_Kq_{K,j}$ of the local plane waves, and $\lambda_{K,j}$ measures the squared size of its Cauchy trace in the metric $m_K$.  If $\lambda_{K,j}$ is small, this combination has very small Dirichlet and normal traces on $\partial K$ and contributes little independent information to the global skeleton problem.  Since $M_K$ is Hermitian, the eigenvectors are chosen Euclidean-orthonormal, so $Q_K^*Q_K=I$.  In particular, $Q_{K,r}^*Q_{K,r}=I$, whereas $Q_{K,r}^*M_KQ_{K,r}=\Lambda_{K,r}$.  Thus the retained eigenvectors are orthonormal in coefficient space but are not yet orthonormal in the trace metric.  Multiplication by $\Lambda_{K,r}^{-1/2}$ converts them to an $m_K$-orthonormal trace basis.  We retain the modes satisfying
\begin{equation}\label{eq:rankrule}
 \lambda_{K,j}>\tau_{\rm rank}\lambda_{K,1},
 \qquad \tau_{\rm rank}=10^{-12}
\end{equation}
in the reported calculations, and set $C_K=Q_{K,r}\Lambda_{K,r}^{-1/2}$.  Then the retained local trace basis $\Phi_KC_K$ is $m_K$-orthonormal.  With $C=\operatorname{diag}(C_K)$ the compressed Galerkin matrices are
\begin{equation}\label{eq:compressed-system}
 A_c=C^*AC,\qquad f_c=C^*f,\qquad G_c=C^*GC,
\end{equation}
and the graph--Riesz solve is applied to $(A_c,f_c,G_c)$.

It is important that this compression changes the retained Galerkin space; it is not only a rescaling of coordinates.  If $c_\perp$ lies in the discarded eigenspace, then
\begin{equation}\label{eq:discard-bound}
 m_K(\Phi_Kc_\perp,\Phi_Kc_\perp)
 \le \tau_{\rm rank}\lambda_{K,1}\|c_\perp\|_2^2.
\end{equation}
Thus $r_K$ denotes the nominal local dimension and $r_K^{\rm eff}$ the retained trace rank.  In the double-precision calculations below we initially take $\tau_{\rm rank}=10^{-12}$.  This is only a numerical cutoff and is not part of the definition of the method.  In Section~\ref{subsec:dtn-precision} we decrease this value and also recompute the trace spectrum at higher precision.  Those experiments show that useful high-$p$ trace modes may lie below $10^{-12}$, so the cutoff must be chosen together with the working precision and the accuracy required from the calculation.

\section{Direction Optimization by the Skeleton Residual}\label{sec:nonlinear}

\subsection{The skeleton residual}\label{subsec:exactJ}
We now introduce the residual that will be used in both direction optimization problems.  For a coefficient vector $c$ and a direction state $Z$, let
\[
 v_h(c,Z)=\sum_{j=1}^{N_h}c_j\varphi_j(Z)
\]
be the associated broken Trefftz field.  For the canonical constant-wavenumber model \eqref{eq:model}, define the direction objective
\begin{align}\label{eq:J}
 \mathcal J(c,Z)=&\sum_{e\in\Ei}\int_e
 \left[
 \alpha\kappa\big|\jump{v_h(c,Z)}\big|^2
 +\frac{\beta}{\kappa}\big|\jumpn{\nabla_hv_h(c,Z)}\big|^2
 \right]ds\\
 &+\alpha\kappa\int_{\Gamma_D}|v_h(c,Z)-g|^2\,ds\nonumber\\
 &+\frac{1}{2\kappa}\int_{\Gamma_R}
 \big|\partial_nv_h(c,Z)+i\kappa v_h(c,Z)\big|^2\,ds.\nonumber
\end{align}
No volume residual is needed because $v_h(c,Z)$ satisfies the Helmholtz equation in every element.  The interior terms measure the jumps of the field and normal flux, the Dirichlet term measures the boundary mismatch, and the final term measures the impedance condition on $\Gamma_R$.  For piecewise-constant transmission experiments with $\Gamma_R=\varnothing$, the first two interior weights are replaced by $\alpha\xi_e$ and $\beta/\xi_e$, with $\xi_e$ from \eqref{eq:xi-interface}, exactly as in the implemented PWDG form.

The factor $(2\kappa)^{-1}$ on $\Gamma_R$ is a least-squares scaling, not a PWDG flux parameter.  Indeed,
\begin{equation}\label{eq:impedance-weight-expansion}
 \frac{1}{2\kappa}|\partial_n v+i\kappa v|^2
 =\frac{1}{2\kappa}|\partial_n v|^2
 +\frac{\kappa}{2}|v|^2
 +\operatorname{Im}(\partial_n v\,\overline v).
\end{equation}
For the symmetric Robin split $\delta=1/2$, the first two diagonal terms have the same coefficients as the two separate Robin traces in $\mathcal G_h$, but the cross term in \eqref{eq:impedance-weight-expansion} remains.  Thus $\mathcal J$ and $\mathcal G_h$ agree term by term only when $\Gamma_R=\varnothing$.  On $\Gamma_R$, $\mathcal J$ measures the strong impedance mismatch while $\mathcal G_h$ is the coercive PWDG trace metric.

When $\Gamma_R=\varnothing$, as in the manufactured square-domain examples, the relation between the residual and the DG/graph norm is simpler and in fact exact.

\begin{proposition}[exact graph-error identity for the all-Dirichlet specialization]\label{prop:exactJ}
Assume $\Gamma_R=\varnothing$ and let $u$ be the exact solution, with the usual continuity of $u$ and normal flux across any material interfaces \cite{MoiolaSpence2019}.  Then for every broken Trefftz field $v_h(c,Z)$,
\begin{equation}\label{eq:exactJ}
 {\mathcal J(c,Z)=\mathcal G_h(u-v_h(c,Z),u-v_h(c,Z))
 =\|u-v_h(c,Z)\|_{\mathcal G,h}^2.}
\end{equation}
The identity is independent of how $c$ is obtained.  It requires, of course, that the same positive weights $\alpha$ and $\beta$ be used in $\mathcal J$ and $\mathcal G_h$.  This is how both quantities are defined here.  All all-Dirichlet computations use the common choice $\alpha=\beta=1/2$.
\end{proposition}
\begin{proof}
Across an interior edge,
\[
 \jump{u-v_h}=-\jump{v_h},\qquad
 \jumpn{\nabla(u-v_h)}=-\jumpn{\nabla_hv_h},
\]
because the exact field and its normal flux are continuous.  On the Dirichlet boundary $u=g$, so $u-v_h=g-v_h$.  Substitution into \eqref{eq:graphform} with $\Gamma_R=\varnothing$ gives \eqref{eq:exactJ} term by term.
\end{proof}

All terms in Proposition~\ref{prop:exactJ} are computed directly from edge integrals for the plane-wave spaces used here.  No pointwise sampling of the interior edges is required.

\subsection{Residual control for mixed Dirichlet--impedance boundaries}\label{subsec:mixed-control}
For the mixed problem, the impedance residual contains a cross term and $\mathcal J$ is no longer identical to the PWDG graph norm.  We can still obtain an $L^2$ error bound provided a stability estimate is available for broken Trefftz fields.  We state this requirement explicitly.

\begin{assumption}[broken Trefftz stability for the mixed problem]\label{ass:mixed-stability}
Let $w$ be elementwise Helmholtz on $\mathcal T_h$.  There is a constant $C_{\rm stab}(\kappa,\Omega)$, independent of the particular broken representation of $w$, such that
\begin{equation}\label{eq:mixed-stability}
\begin{aligned}
 \|w\|_{L^2(\Omega)} \le C_{\rm stab}(\kappa,\Omega)\Big(&
 \sum_{e\in\mathcal E_h^I}\big[\kappa\|\jump{w}\|_{L^2(e)}^2
 +\kappa^{-1}\|\jumpn{\nabla_h w}\|_{L^2(e)}^2\big]\\
 &+\kappa\|w\|_{L^2(\Gamma_D)}^2
 +\kappa^{-1}\|\partial_n w+i\kappa w\|_{L^2(\Gamma_R)}^2\Big)^{1/2}.
\end{aligned}
\end{equation}
\end{assumption}
Assumption~\ref{ass:mixed-stability} plays the same role as a Helmholtz stability estimate in a standard residual argument.  Continuous impedance estimates give explicit information on the dependence on $\kappa$ and the geometry \cite{BaskinSpenceWunsch2016}.  We state the broken estimate as an assumption because the form above is not proved here as a direct consequence of the continuous result.

\begin{proposition}[mixed-boundary residual control]\label{prop:mixed-control}
Under Assumption~\ref{ass:mixed-stability}, let $u$ solve \eqref{eq:model} and let $v_h$ be any broken Trefftz field.  Then
\begin{equation}\label{eq:mixed-J-bound}
 {\|u-v_h\|_{L^2(\Omega)}\le C_J(\alpha,\beta)\,
 C_{\rm stab}(\kappa,\Omega)\,\mathcal J(v_h)^{1/2},}
\end{equation}
where
\begin{equation}\label{eq:CJ-explicit}
 C_J(\alpha,\beta)
 :=\sqrt{\max\{\alpha^{-1},\beta^{-1},2\}}.
\end{equation}
In particular, $C_J(1/2,1/2)=\sqrt2$ for the weights used in the numerical experiments.
\end{proposition}
\begin{proof}
Set $w=u-v_h$.  The exact field has zero interior jumps, satisfies $u=g$ on $\Gamma_D$, and satisfies $\partial_nu+i\kappa u=0$ on $\Gamma_R$.  Hence every trace residual of $w$ is, up to sign, exactly the corresponding term appearing in \eqref{eq:J}.  Comparing the four groups in \eqref{eq:mixed-stability} with \eqref{eq:J} gives the factors $\alpha^{-1}$, $\beta^{-1}$, $\alpha^{-1}$, and $2$, respectively.  Therefore the unweighted quantity in \eqref{eq:mixed-stability} is bounded by $C_J(\alpha,\beta)\mathcal J(v_h)^{1/2}$ with \eqref{eq:CJ-explicit}, which proves the result.
\end{proof}

For the all-Dirichlet problem we therefore have the identity \eqref{eq:exactJ}, whereas for the mixed problem we obtain the upper bound \eqref{eq:mixed-J-bound}.  We do not prove a converse efficiency estimate.  The estimate is nevertheless sufficient to justify using $\mathcal J$ as a computable error quantity for the mixed problem, with the qualification that the bound may become less sharp when the Helmholtz stability constant is large.

\subsection{Evaluation on straight edges and local contributions}\label{subsec:exact-edge}
We next describe the computation of the residual matrices on a straight edge.  Let the edge be parameterized by
\[
 x(t)=a+t\ell,\qquad 0\le t\le1,\qquad |e|=|\ell|.
\]
For two local Trefftz functions
\[
 \phi_j(x)=e^{iq_j\cdot(x-x_{K_j})},\qquad
 \phi_i(x)=e^{iq_i\cdot(x-x_{K_i})},
\]
including complex directions, one has the closed formula
\begin{equation}\label{eq:edge-exp-integral}
 \int_e\phi_j\overline{\phi_i}\,ds
 =|e|\,e^{iq_j\cdot(a-x_{K_j})-i\overline{q_i}\cdot(a-x_{K_i})}
 \Psi\!\left(i(q_j-\overline{q_i})\cdot\ell\right),
\end{equation}
where
\[
 \Psi(\zeta)=\frac{e^\zeta-1}{\zeta},\qquad \Psi(0)=1.
\]
In the implementation we evaluate $\Psi(\zeta)$ as $\operatorname{expm1}(\zeta)/\zeta$ when $\zeta\ne0$ and set $\Psi(0)=1$.  Normal derivatives introduce only the constant factors $iq_j\cdot n$ and $-i\overline{q_i}\cdot n$.  Hence all basis--basis contributions from interior edges and homogeneous straight boundary edges are available in closed form.  Differentiating with respect to a direction parameter only adds the first moment $\int_0^1t e^{\zeta t}\,dt$, so the corresponding matrix derivatives are computed in the same way.

Numerical quadrature is used only for data-dependent terms for which no closed expression is available, for example
\[
 \int_{\Gamma_D}g\,\overline{\phi_i}\,ds,
 \qquad
 \int_{\Gamma_D}|g|^2\,ds,
\]
when $g$ has no convenient closed form.  For manufactured plane-wave boundary data these terms may also be evaluated analytically.  The interior residual matrix is therefore assembled from exact edge moments rather than from sampled residual values.

For fixed directions, expansion of \eqref{eq:J} gives the Hermitian quadratic
\begin{equation}\label{eq:Jquadratic}
 {\mathcal J(c,Z)=c^*M(Z)c-2\Re\{c^*d(Z)\}+\gamma.}
\end{equation}
The weights in \eqref{eq:J} are retained in each part of this expansion.  With
$v_h=\sum_j c_j\varphi_j$, the entries of the Gram matrix for the canonical
constant-wavenumber problem are
\begin{align}\label{eq:M-explicit}
 M_{ij}(Z)={}&
 \sum_{e\in\Ei}\left[
 \alpha\kappa\int_e \jump{\varphi_j}\cdot\overline{\jump{\varphi_i}}\,ds
 +\frac{\beta}{\kappa}\int_e
 \jumpn{\nabla_h\varphi_j}\,\overline{\jumpn{\nabla_h\varphi_i}}\,ds
 \right] \nonumber\\
 &+\alpha\kappa\int_{\Gamma_D}\varphi_j\overline{\varphi_i}\,ds \nonumber\\
 &+\frac{1}{2\kappa}\int_{\Gamma_R}
 \big(\partial_n\varphi_j+i\kappa\varphi_j\big)
 \overline{\big(\partial_n\varphi_i+i\kappa\varphi_i\big)}\,ds .
\end{align}
The data terms are
\begin{equation}\label{eq:d-gamma-explicit}
 d_i(Z)=\alpha\kappa\int_{\Gamma_D}g\,\overline{\varphi_i}\,ds,
 \qquad
 \gamma=\alpha\kappa\int_{\Gamma_D}|g|^2\,ds.
\end{equation}
Thus $\alpha$ enters $M$, $d$, and $\gamma$, while $\beta$ enters the
normal-flux block of $M$.  The parameter $\delta$ does not occur in
\eqref{eq:M-explicit}: on $\Gamma_R$ the optimization functional uses the
impedance residual in \eqref{eq:J}, rather than the separately weighted
Dirichlet and Neumann traces used in the PWDG graph form.  For the
piecewise-constant transmission problem, the two interior weights in
\eqref{eq:M-explicit} are replaced by $\alpha\xi_e$ and $\beta/\xi_e$.
In the all-Dirichlet specialization, $M(Z)$ is precisely the matrix of the
DG/graph inner product on the current Trefftz space.  For the mixed
Dirichlet--impedance problem it is instead the Gram matrix of the residual
functional \eqref{eq:J}, and should not be confused with the PWDG graph matrix $G$.

For the all-Dirichlet problem, Proposition~\ref{prop:exactJ} also gives an elementwise error decomposition.  Write
\begin{align}\label{eq:edge-indicator}
 &\eta_e^2=\alpha\xi_e\|\jump{v_h}\|_{L^2(e)}^2
 +\frac{\beta}{\xi_e}\|\jumpn{\nabla_hv_h}\|_{L^2(e)}^2,
 \qquad e\in\mathcal E_h^I,\\
 &\eta_e^2=\alpha\kappa\|v_h-g\|_{L^2(e)}^2,
 \qquad e\subset\Gamma_D.
\end{align}
Assign half of each interior contribution to either adjacent element and the full boundary contribution to its element:
\begin{equation}\label{eq:element-indicator}
 \eta_K^2=\frac12\sum_{e\subset\partial K\cap\mathcal E_h^I}\eta_e^2
 +\sum_{e\subset\partial K\cap\Gamma_D}\eta_e^2.
\end{equation}
Then
\begin{equation}\label{eq:exact-estimator}
 {\sum_K\eta_K^2=\mathcal J(c,Z)=\|u-v_h\|_{\mathcal G,h}^2.}
\end{equation}
Thus, for these tests, the numbers $\eta_K$ localize the exact DG/graph error.  We use them first to decide which elements require attention and, when only mesh refinement is allowed, in the usual D\"orfler marking procedure \cite{Dorfler1996}.  Residual based PWDG refinement is studied in \cite{KapitaMonkWarburton2015,CongreveHoustonPerugia2019}.

\subsection{Part A: direction-optimized PWDG as a constrained minimization problem}\label{subsec:pwdg-constrained}
We first consider the formulation in which the PWDG equations are retained during the direction search.  For each admissible direction state $Z$, let $c_{\rm PW}(Z)$ be the coefficient vector obtained from the fixed-direction PWDG problem,
\begin{equation}\label{eq:cpw}
 A(Z)c_{\rm PW}(Z)=f(Z),
\end{equation}
computed in practice by graph--Riesz normalization.  Define the reduced continuous PWDG skeleton-residual objective
\begin{equation}\label{eq:Jpwdg}
 \widehat{\mathcal J}_{\rm PW}(Z)
 :=\mathcal J(c_{\rm PW}(Z),Z).
\end{equation}
\begin{equation}\label{eq:directionopt}
 Z_{\rm PW}^\star\in
 \operatorname*{arg\,min}_{Z\in\mathcal Z_h(\bm r,\eta_{\max})}
 \widehat{\mathcal J}_{\rm PW}(Z),
\end{equation}
Changes of local graph rank are treated as outer active-set events.  Equivalently, before compression events are introduced, it is the equality-constrained problem
\begin{equation}\label{eq:constrained-main}
 {
 \min_{c,Z}\mathcal J(c,Z)
 \quad\text{subject to}\quad A(Z)c=f(Z).}
\end{equation}
Equation \eqref{eq:constrained-main} has the form of a finite-dimensional equality-constrained optimization problem.  Following the real-angle PWDG control formulation of Agrawal and Hoppe \cite{AgrawalHoppe2017}, we introduce a complex Lagrange multiplier $\mu$.  The directions are now allowed to be complex and the objective is the skeleton residual.  For real direction coordinates $\rho$, define
\begin{equation}\label{eq:partA-lagrangian}
 \mathcal L(c,\rho,\mu)
 =\mathcal J(c,\rho)+2\Re\!\left[\mu^*\big(A(\rho)c-f(\rho)\big)\right].
\end{equation}
On a fixed-rank stratum, the first-order constrained stationarity conditions are
\begin{align}
 &A(\rho)c=f(\rho), \label{eq:partA-kkt-constraint}\\
 &\partial_{\bar c}\mathcal J+A(\rho)^*\mu=0, \label{eq:partA-kkt-c}\\
 &\partial_{\rho_\ell}\mathcal J\big|_c
 +2\Re\!\left[\mu^*\big(A_{\rho_\ell}c-f_{\rho_\ell}\big)\right]=0
 \qquad(\ell=1,\ldots,m). \label{eq:partA-kkt-rho}
\end{align}
Since $c$ is constrained by the PWDG equations, $\partial_{\bar c}\mathcal J$ is not required to vanish.  Instead, \eqref{eq:partA-kkt-c} states that this gradient is balanced by $A^*\mu$.  We later use the equivalent adjoint variable $\lambda=-\mu$ when evaluating the reduced derivatives.

The multiplier can also be interpreted as the sensitivity to a perturbation of the discrete PWDG equation.  Suppose that
\[
 A(\rho)c=f(\rho)+\varepsilon r.
\]
Differentiating the constrained objective with respect to $\varepsilon$ gives the dual pairing with $\mu$, with the sign determined by \eqref{eq:partA-lagrangian}.  Thus $\mu$ measures the change in the optimized residual produced by a defect in the discrete state equation.  It is an adjoint variable and is not another approximation of the Helmholtz solution.  This point is also useful when comparing the two formulations.  In Part~B the coefficients are unconstrained least-squares variables and satisfy $\partial_{\bar c}\mathcal J=0$, while in Part~A this gradient is balanced by $A^*\mu$.  If the residual is zero and $A$ is nonsingular, then $\partial_{\bar c}\mathcal J=0$ and consequently $\mu=\lambda=0$.

The bounds on the angles do not by themselves guarantee that $Z\mapsto c_{\rm PW}(Z)$ is well defined for every admissible point.  For example, two local directions can coalesce and make the coefficient representation singular.  For existence statements we therefore restrict to a rank-safe closed subset
\begin{equation}\label{eq:Zdelta}
 \mathcal Z_h^\delta
 :=\{Z\in\mathcal Z_h(\bm r,\eta_{\max}):\lambda_{\min}(G(Z))\ge\delta\},
 \qquad \delta>0,
\end{equation}
with a fixed active basis.  On $\mathcal Z_h^\delta$, $A(Z)$ is nonsingular, $c_{\rm PW}(Z)$ and $\widehat{\mathcal J}_{\rm PW}(Z)$ are continuous, and a minimizer exists by compactness.  In the actual algorithm, approaching a rank-deficient configuration triggers local trace compression rather than an attempt to differentiate through the rank change.

In the complex-angle transmission calculation we add $10^{-12}\sum_j\eta_j^2$ to the nonlinear objective.  This very small term is used only to select between residual-equivalent representations.  It is not included in the reported values of $\mathcal J$.

\subsection{Local analysis for fixed rank}\label{subsec:local-theory}
We now give a local result for the reduced direction problem.  There are two reasons for restricting the statement to a fixed-rank neighborhood.  A permutation of the rays on one element does not change the local Trefftz space, so the ordered direction vector is not uniquely determined.  Also, if two rays coalesce or a trace becomes redundant, the dimension of the retained space changes.  We therefore fix the active rank and measure the distance to the finite set of permutations of the target direction state.

We assume that the residual determines the directions locally.  More precisely, if $R(\rho)$ denotes a smooth residual coordinate vector, we require $DR(\rho_\star)$ to have full column rank after one representative of the permutation orbit has been fixed.  Equivalently, $DR(\rho_\star)^*DR(\rho_\star)$ is positive definite.  This assumption rules out a nonzero first-order change of the directions that leaves the residual unchanged.  It is not satisfied when rays coalesce or become redundant, which is why those cases are treated separately by the rank reduction.

Let $\rho\in\mathbb R^m$ collect the free real direction coordinates and define
\[
 U(\rho):=v_h(c_{\rm PW}(\rho),\rho).
\]
Local ray permutations represent the same Trefftz space.  We therefore measure distance to the finite permutation orbit $\mathcal O(\rho_\star)$ and work on a neighborhood in which the active trace rank is fixed.

\begin{theorem}[local quadratic growth and direction-to-field stability]\label{thm:local-growth}
Let $\rho_\star$ satisfy $\widehat{\mathcal J}_{\rm PW}(\rho_\star)=0$.  Assume that, in a neighborhood of $\rho_\star$, the active graph rank is constant, $A(\rho)$ is uniformly invertible, and a smooth residual coordinate map $R$ with
\[
 \widehat{\mathcal J}_{\rm PW}(\rho)=\|R(\rho)\|_2^2
\]
has full-column-rank Jacobian $DR(\rho_\star)$.  Then there are constants $c_0,C_0,C_U>0$ such that
\begin{equation}\label{eq:local-growth}
 c_0\,\operatorname{dist}(\rho,\mathcal O(\rho_\star))^2
 \le \widehat{\mathcal J}_{\rm PW}(\rho)
 \le C_0\,\operatorname{dist}(\rho,\mathcal O(\rho_\star))^2,
\end{equation}
and
\begin{equation}\label{eq:direction-field-lipschitz}
 \|U(\rho)-U(\rho_\star)\|_{\mathcal G,h}
 \le C_U\,\operatorname{dist}(\rho,\mathcal O(\rho_\star))
\end{equation}
for $\rho$ sufficiently close to $\rho_\star$.  Consequently, standard Gauss--Newton or trust-region/Levenberg--Marquardt iteration is locally convergent, with the usual quadratic zero-residual Gauss--Newton behavior once the full-rank neighborhood is entered \cite{More1978}.
\end{theorem}
\begin{proof}
With the active rank fixed, the plane-wave traces depend smoothly on $\rho$, and consequently so do $A(\rho)$ and $f(\rho)$.  The assumed uniform invertibility of $A(\rho)$ then gives a smooth coefficient map $c_{\rm PW}(\rho)=A(\rho)^{-1}f(\rho)$.  Hence the field $U(\rho)$ and residual vector $R(\rho)$ are also smooth.

Choose the nearest representative $\widetilde\rho_\star$ in the finite permutation orbit $\mathcal O(\rho_\star)$ and set $d=\rho-\widetilde\rho_\star$.  Because $R(\widetilde\rho_\star)=0$, Taylor expansion gives
\[
 R(\rho)=DR(\widetilde\rho_\star)d+\mathcal O(|d|^2).
\]
Since $DR(\widetilde\rho_\star)$ has full column rank, its smallest and largest singular values give positive lower and upper bounds for the linear term.  After reducing the neighborhood, the quadratic remainder can be absorbed into these bounds and we obtain $c_0|d|^2\leq\|R(\rho)\|_2^2\leq C_0|d|^2$, which proves \eqref{eq:local-growth}.  Applying a first-order Taylor estimate to $U$ gives \eqref{eq:direction-field-lipschitz}.  The final statement follows from the usual local theory for full-rank zero-residual least squares \cite{More1978}.
\end{proof}

If rays coalesce or become redundant, the fixed-rank hypothesis is no longer satisfied.  In the algorithm this is treated by recomputing the local trace rank as described in Section~\ref{subsec:compression}; we do not differentiate through the change of rank.

\subsection{Relation to the Monk--Wang least-squares method}\label{subsec:monkwang}
Monk and Wang \cite{MonkWang1999} prescribe discontinuous local Helmholtz solutions and minimize a quadratic skeleton functional over their coefficients.  Their interior term uses the full gradient jump, whereas \eqref{eq:J} uses the normal-flux DG jump aligned with the PWDG graph form.  Here a different, purely oriented notation is useful.  Fix one unit normal $n_e$ and tangent $t_e$ on a straight edge and define the ordinary scalar difference
\[
 \sjump{v}:=v^+-v^-.
\]
Then $\sjump{\cdot}$ depends on this chosen orientation, unlike the DG double-line jump above, and
\[
 \sjump{\nabla v}=\sjump{\partial_{n_e}v}\,n_e+\partial_{t_e}\sjump{v}\,t_e.
\]
The residual normalization in \eqref{eq:J} is chosen so that this comparison has no additional arbitrary constants.  For constant $\kappa$, $\alpha=\beta=1/2$, and matching impedance sign conventions, multiplication of \eqref{eq:J} by $2\kappa$ gives unit weight on the normal-flux jump and impedance residual and weight $\kappa^2$ on the solution jump and Dirichlet residual.  Decomposing the full gradient jump into its normal and tangential parts therefore yields
\begin{equation}\label{eq:MWrelation}
 {\mathcal J_{\rm MW}(c,Z)
 =2\kappa\,\mathcal J(c,Z)
 +\sum_{e\in\Ei}\|\partial_{t_e}\sjump{v_h(c,Z)}\|_{L^2(e)}^2.}
\end{equation}
With this scaling, the relation between the two functionals is given exactly by \eqref{eq:MWrelation}.  The present functional retains the normal component of the gradient jump, whereas the Monk--Wang functional also contains the tangential derivative of the solution jump.  It should therefore be viewed as a normal-flux version of the Monk--Wang residual, and not as the same functional.

For fixed directions, the coefficient least-squares solution is therefore characterized directly by the exact Hermitian system
\begin{equation}\label{eq:cls}
 M(Z)c_{\rm LS}(Z)=d(Z),
 \qquad
 c_{\rm LS}(Z)=\operatorname*{arg\,min}_{c\in\C^{N_h}}\mathcal J(c,Z).
\end{equation}
In Part~B we keep the coefficient projection \eqref{eq:cls} but allow its local directions to vary.  For every fixed $Z$,
\begin{equation}\label{eq:LSdominatesJ}
 \mathcal J(c_{\rm LS}(Z),Z)
 \le \mathcal J(c_{\rm PW}(Z),Z),
\end{equation}
but this residual ordering need not agree with the $L^2$-error ordering.

For Part~A, the equality constraint $A(Z)c=f(Z)$ may be expressed with
\[
 \mathcal L(c,Z,\lambda)
 =\mathcal J(c,Z)-2\Re\{\lambda^*(A(Z)c-f(Z))\}.
\]
Stationarity in $c$ gives the adjoint equation $A^*\lambda=\partial_{\bar c}\mathcal J$.  The resulting reduced direction derivative is given below.  The adjoint term is present because the Part~A coefficients must remain on the PWDG solution manifold; it is absent from the free coefficient minimization in Part~B.

\subsection{Differentiation with respect to the directions}\label{subsec:wirtinger}
To differentiate with respect to a complex angle we use the Wirtinger derivatives \cite{HjorungnesGesbert2007}, in which $z$ and $\overline z$ are treated as independent variables for the purpose of differentiation.  For $z=\theta+i\eta$,
\[
 z=\theta+i\eta,
 \qquad
 \partial_z=\frac12(\partial_\theta-i\partial_\eta),
 \qquad
 \partial_{\overline z}=\frac12(\partial_\theta+i\partial_\eta).
\]
If $J$ is real valued, then $\partial_{\overline z}J=\overline{\partial_zJ}$.  The notation is particularly convenient here because a trial plane wave depends holomorphically on $z$, while the conjugated plane wave in the test position depends on $\overline z$.
For
\[
 q(z)=\kappa(\cos z,\sin z),
 \qquad
 q'(z)=\kappa(-\sin z,\cos z),
\]
the local Trefftz basis function
\[
 \phi_z(x)=\exp\!\big(iq(z)\cdot(x-x_K)\big)
\]
is entire in $z$.  Hence
\begin{equation}\label{eq:phi-wirt}
 \partial_z\phi_z
 =i\,q'(z)\cdot(x-x_K)\,\phi_z,
 \qquad
 \partial_{\overline z}\phi_z=0.
\end{equation}
The conjugated test function satisfies
\begin{equation}\label{eq:phibar-wirt}
 \partial_z\overline{\phi_z}=0,
 \qquad
 \partial_{\overline z}\overline{\phi_z}
 =-i\,\overline{q'(z)}\cdot(x-x_K)\,\overline{\phi_z}.
\end{equation}
The gradient and normal-trace derivatives follow by the product rule:
\begin{align}
 &\partial_z\nabla\phi_z
 =i q'(z)\phi_z+i q(z)\,\partial_z\phi_z\nonumber\\
 &\hspace{2em}=\Big(iq'(z)-q(z)\,[q'(z)\cdot(x-x_K)]\Big)\phi_z,\label{eq:gradphi-wirt}\\
 &\partial_z\partial_n\phi_z
 =i(q'(z)\cdot n)\phi_z+i(q(z)\cdot n)\partial_z\phi_z.\label{eq:nphi-wirt}
\end{align}
These tangent functions remain Trefftz because $\partial_z$ commutes with $-\Delta-\kappa^2$.

Although each plane wave is holomorphic in its own angle, the Galerkin matrix depends on both $Z$ and $\overline Z$, since $a_h$ is conjugate-linear in the test argument.  For independent direction parameters and
\[
 A_{ij}(Z,\overline Z)=a_h(\phi_{z_j},\phi_{z_i}),
 \qquad
 f_i(Z,\overline Z)=F_h(\phi_{z_i}),
\]
we obtain
\begin{align}
 &\partial_{z_\ell}A_{ij}
 =\delta_{j\ell}\,a_h(\partial_z\phi_{z_j},\phi_{z_i}),
 \label{eq:Az}\\
 &\partial_{\overline z_\ell}A_{ij}
 =\delta_{i\ell}\,a_h(\phi_{z_j},\partial_z\phi_{z_i}),
 \label{eq:Abarz}\\
 &\partial_{z_\ell}f_i=0,
 \qquad
 \partial_{\overline z_\ell}f_i
 =\delta_{i\ell}F_h(\partial_z\phi_{z_i}).
 \label{eq:f-wirt}
\end{align}
For shared direction parameters these formulas are summed over the corresponding basis functions.  Equation \eqref{eq:Az} differentiates the trial column associated with $z_\ell$, whereas \eqref{eq:Abarz} differentiates the corresponding test row.

At fixed active rank,
\[
 A(Z,\overline Z)c(Z,\overline Z)=f(Z,\overline Z).
\]
Differentiation gives two coefficient sensitivities per complex angle,
\begin{align}
 &A c_{z_\ell}
 =-A_{z_\ell}c,
 \label{eq:cz}\\
 &A c_{\overline z_\ell}
 =f_{\overline z_\ell}-A_{\overline z_\ell}c.
 \label{eq:cbarz}
\end{align}
The real sensitivities are
\begin{equation}\label{eq:real-from-wirt}
 c_{\theta_\ell}=c_{z_\ell}+c_{\overline z_\ell},
 \qquad
 c_{\eta_\ell}=i(c_{z_\ell}-c_{\overline z_\ell}).
\end{equation}
More generally, for either real parameter $\rho_\ell\in\{\theta_\ell,\eta_\ell\}$,
\begin{equation}\label{eq:sensitivity}
 A\,c_{\rho_\ell}
 =f_{\rho_\ell}-A_{\rho_\ell}c.
\end{equation}
The same factorization of $A$ may be used for all the sensitivity right-hand sides.  We differentiate the physical system $Ac=f$, and hence do not need to differentiate the graph-coordinate map $B$.  During one smooth direction step the compression matrix $C$ is fixed.  If the local rank changes, the active space is recomputed before the next step.

For the exact skeleton quadratic \eqref{eq:Jquadratic}, let $M_\rho$ and $d_\rho$ denote the derivatives induced by a real direction parameter $\rho\in\{\theta_\ell,\eta_\ell\}$.  Combining the product rule with the coefficient sensitivity \eqref{eq:sensitivity} gives the exact reduced derivative
\begin{equation}\label{eq:Jexact-gradient}
 {
 \frac{d\widehat{\mathcal J}_{\rm PW}}{d\rho}
 =2\Re\!\left[c_\rho^*(Mc-d)\right]
 +c^*M_\rho c-2\Re\!\left(c^*d_\rho\right).}
\end{equation}
Interior terms in $M_\rho$ follow by differentiating the closed edge moment \eqref{eq:edge-exp-integral}; only nonanalytic boundary-data terms require quadrature.

If many direction variables are present, solving \eqref{eq:sensitivity} separately for every parameter is unnecessary.  An adjoint equation gives all derivatives after one additional global solve.  If $g_c=\partial\mathcal J/\partial\overline c$, the Part~A stationarity equation \eqref{eq:partA-kkt-c} gives $A^*\mu=-g_c$.  With $\lambda=-\mu$,
\begin{equation}\label{eq:adjoint}
 A^*\lambda=g_c.
\end{equation}
then for any real direction parameter $\rho_\ell$ the reduced derivative can be written
\begin{equation}\label{eq:adjoint-gradient}
 \frac{d\mathcal J}{d\rho_\ell}
 =\left.\frac{\partial\mathcal J}{\partial\rho_\ell}\right|_{c}
 +2\Re\!\left[\lambda^*\big(f_{\rho_\ell}-A_{\rho_\ell}c\big)\right].
\end{equation}
After $\lambda$ has been computed, the derivative with respect to every local direction parameter is obtained from the corresponding local matrix and load derivatives.

\subsection{Part B: Trefftz least squares with variable directions}\label{subsec:jointls}
In the second formulation we remove the PWDG state equation and minimize the skeleton residual directly.  This is related to the wave-tracking least-squares approach of Amara et al. \cite{AmaraChaudhryDiazDjellouliFiedler2014}, except that we eliminate the linear coefficients by variable projection.  The problem is
\begin{equation}\label{eq:fullLS}
 {(c^\star,Z^\star)\in\operatorname*{arg\,min}_{c,Z}\mathcal J(c,Z).}
\end{equation}
For fixed $Z$, \eqref{eq:Jquadratic} is quadratic in $c$ and gives the usual linear Trefftz least-squares problem.  The difference from Monk--Wang is that the local Trefftz functions are no longer fixed, since their directions are also unknown.

\medskip\noindent\textit{Full-space first-order conditions.}
Because $M=M^*$, Wirtinger differentiation with respect to the free coefficient vector gives
\begin{equation}\label{eq:joint-c-normal}
 {\partial_{\bar c}\mathcal J=Mc-d.}
\end{equation}
For a real direction parameter $\rho$, the product rule gives, with $c$ held fixed,
\begin{equation}\label{eq:joint-z-normal}
 {\partial_\rho\mathcal J
 =c^*M_\rho c-2\Re(c^*d_\rho).}
\end{equation}
Equivalently, in complex-angle coordinates the $\bar z_\ell$ equation is obtained by differentiating the test-side factors in the exact edge moments.  The underlying basis derivative remains
\[
 \partial_z\phi_z=i\,q'(z)\cdot(x-x_K)\phi_z,
\]
so \eqref{eq:joint-z-normal} is exactly the product-rule condition expressing orthogonality of the residual functional to the directional tangent generated by $c_j\partial_z\phi_{z_j}$.  A joint stationary point therefore satisfies the coefficient projection equation and all directional tangent equations simultaneously.

\medskip\noindent\textit{Variable projection.}
Since the dependence on $c$ is linear, it is unnecessary to include the coefficients in the nonlinear iteration.  We use the variable-projection reduction of Golub and Pereyra \cite{GolubPereyra1973}.  On a fixed-rank stratum,
\begin{equation}\label{eq:vp-c}
 {M(Z)c_{\rm LS}(Z)=d(Z),}
\end{equation}
with the minimum-norm solution used if the active matrix is singular.  When $M$ is positive definite,
\begin{equation}\label{eq:vp-J}
 {
 \widehat{\mathcal J}_{\rm LS}(Z)
 =\gamma-d(Z)^*M(Z)^{-1}d(Z).}
\end{equation}
The reduced and full minimization problems have the same minimum, while the nonlinear problem contains only the direction variables.  The envelope derivative contains no coefficient sensitivity:
\begin{equation}\label{eq:vp-envelope}
 {
 \frac{d\widehat{\mathcal J}_{\rm LS}}{d\rho}
 =c_{\rm LS}^*M_\rho c_{\rm LS}
 -2\Re(c_{\rm LS}^*d_\rho).}
\end{equation}
Equation \eqref{eq:vp-envelope} is used in the exact-edge implementation.  When the residual is very small, we evaluate the objective from the positive edge contributions in \eqref{eq:edge-indicator}, since direct evaluation of \eqref{eq:vp-J} can lose accuracy through cancellation.

\medskip\noindent\textit{Krylov realization.}
When the active residual seminorm is a norm on the retained Trefftz space, $M$ is Hermitian positive definite and \eqref{eq:vp-c} can be solved directly by preconditioned conjugate gradients \cite{HestenesStiefel1952,MonkWang1999}, without normal equations.  Trace-rank compression from Section~\ref{subsec:compression} is applied if moving directions produce numerical rank loss.  A symmetric eigendecomposition is used as a small-problem fallback.

When several directions are used on each element, optimizing every angle independently can give a large nonlinear problem.  We therefore also use the structured fan
\begin{equation}\label{eq:fan}
 z_{K,j}=\mu_K+s_K\xi_j,
 \qquad j=1,\ldots,r_K,
\end{equation}
with fixed equally spaced centered nodes.  In the reported implementation, $\xi_j=-1+2(j-1)/(r_K-1)$ for $r_K>1$, while $\xi_1=0$ for $r_K=1$.  Thus $s_K$ is the half-width of the angular fan.  The chain rule is applied directly to $M_\rho$ and $d_\rho$.  The nonlinear dimension is then two fan parameters per element, independent of the number of linear coefficients.

\subsection{Frequency continuation}\label{subsec:continuation}
The numerical experiments below show that direct minimization at the target wavenumber can converge to a false local minimum.  To reduce this difficulty we use frequency continuation \cite{Watson1989}.  For
\[
 0<\kappa^{(0)}<\cdots<\kappa^{(M)}=\kappa_{\rm target},
\]
the converged directions at $\kappa^{(m)}$ initialize the solve at $\kappa^{(m+1)}$.  Blind starts at the lowest frequency use Sobol points \cite{Sobol1967}.  The $\omega=12$ transmission tests use $1,2,4,6,8,10,12$.  Only the direction variables are passed from one frequency to the next.  No auxiliary low-frequency field is postprocessed to estimate rays, as is done in the learning procedure of Fang et al. \cite{FangQianZepedaZhao2017}.

\section{Numerical Algorithms}\label{sec:algorithm}
Before giving the numerical results, we summarize the two procedures used in the computations.  Both use the same direction parameterization and the same skeleton residual.  The difference is how the coefficients are obtained.  Part~A solves the PWDG equations for each direction state, whereas Part~B solves the coefficient least-squares problem and removes it from the nonlinear iteration by variable projection.

\begin{algorithm}[H]
\caption{Direction-optimized PWDG (Part A)}
\label{alg:automatic}
\begin{algorithmic}[1]
\State Choose an admissible initial direction state and, when needed, a frequency-continuation schedule.
\For{each continuation frequency}
  \State Assemble the PWDG matrices and compress graph-small local trace modes by \eqref{eq:rankrule}.
  \State Solve the retained PWDG system in graph--Riesz coordinates.
  \State Minimize $\widehat{\mathcal J}_{\rm PW}(Z)$ over the active directions using \eqref{eq:Jexact-gradient}.
  \State Use the converged directions to initialize the next frequency.
\EndFor
\State \Return the PWDG field, directions, effective trace ranks, $\kappa_2(A)$, and $\kappa_{\GR}$.
\end{algorithmic}
\end{algorithm}

\begin{algorithm}[H]
\caption{Exact-edge variable-projection Trefftz least squares (Part B)}
\label{alg:varpro}
\begin{algorithmic}[1]
\State Choose an initial direction state and, when needed, a continuation schedule.
\Repeat
  \State Assemble the exact skeleton matrix $M(Z)$ by \eqref{eq:edge-exp-integral} and assemble $d(Z)$ and $\gamma$ using boundary-data quadrature only when required.
  \State Solve $M(Z)c=d(Z)$ by preconditioned CG on a full-rank active space.
  \State Assemble $M_\rho$ and $d_\rho$ from differentiated exponential moments and evaluate \eqref{eq:vp-envelope}.
  \State Apply a bounded nonlinear step to the direction variables only.
\Until{the objective or direction step satisfies the stopping criterion}
\State \Return $Z$, $c_{\rm LS}(Z)$, and the positive edge sum $\mathcal J=\sum_e\eta_e^2$.
\end{algorithmic}
\end{algorithm}

The nonlinear objective in both algorithms is nonconvex.  In examples where a direct solve at the final wavenumber is unreliable, we use frequency continuation.  We also keep changes of numerical rank outside the smooth direction step: after an accepted state the local trace ranks are recomputed and, if necessary, the active space is changed.

\section{Numerical experiments}\label{sec:numerics}

\subsection{Implementation details and derivative checks}
Unless stated otherwise, all computations in this section are carried out in IEEE double precision on triangular meshes.  For the manufactured square-domain examples we take $\Gamma_R=\varnothing$, so Proposition~\ref{prop:exactJ} gives $\mathcal J$ as the squared DG/graph error.  The plane-wave products on interior edges, including those used in $M$ and in the direction derivatives, are evaluated from \eqref{eq:edge-exp-integral}.  Numerical quadrature is used only for boundary terms for which a closed form is not available.  In the transmission problem the wavenumber is piecewise constant and continuity is imposed across $y=0$; the interior penalties use the arithmetic face scale \eqref{eq:xi-interface}.  We take $\alpha=\beta=1/2$ unless another value is stated.  These parameters do not enter the local rank metric \eqref{eq:localmetric}.  We begin the double-precision calculations with $\tau_{\rm rank}=10^{-12}$ and revisit this choice in the high-$p$ DtN experiment.

We report
\[
 E_{L^2}=\frac{\|u-u_h\|_{L^2(\Omega)}}{\|u\|_{L^2(\Omega)}}.
\]
As a check on the implementation, the analytic direction derivatives were compared with centered finite differences at points that were not stationary.  Differentiation of the exact exponential moments agrees with finite differences at approximately $10^{-9}$--$10^{-10}$ relative accuracy in the transmission tests.  The original PWDG matrix and coefficient-sensitivity checks remain below $6\times10^{-9}$.

\paragraph{Nonlinear solver and constraints.}
Part~A uses trust-region nonlinear least squares with analytic residual derivatives \cite{More1978}.  Angular variables are represented modulo $2\pi$, with $|\eta_j|\le1.8$.  After each accepted state the trace Gramians are recomputed and compression is applied before the next fixed-rank solve.  Part~B uses the same framework after variable projection.  Stopping tests use the scaled residual and step norms.  The reported value of $N_{\rm fev}$ is the number of nonlinear residual evaluations.  Since the derivatives are analytic, no additional finite-difference residual evaluations are included.

\subsection{Three hidden plane waves}\label{subsec:threewave}
We first consider an exact solution consisting of three plane waves,
\begin{equation}\label{eq:threewave}
 u(x)=\sum_{j=1}^3A_j e^{i\kappa d_j\cdot x},
\end{equation}
where the three directions are $17^\circ$, $123^\circ$, and $251^\circ$.  These angles are used to generate the boundary data but are not supplied to the nonlinear solver.  The direction search starts from the same Sobol low-discrepancy angles \cite{Sobol1967}, $43.929^\circ$, $130.738^\circ$, and $269.353^\circ$, for every wavenumber.  Table~\ref{tab:threewave} uses a uniform mesh with 72 triangles and 216 PWDG coefficients.  For all three wavenumbers the nonlinear iteration recovers the three directions to the digits shown.  After this has occurred, the exact solution belongs to the discrete Trefftz space and the remaining $L^2$ error is at the level of roundoff.

\begin{table}[htbp]
\centering
\caption{Three-wave recovery on a uniform mesh with 72 triangles.}
\label{tab:threewave}
\begin{tabular}{|c|r|r|r|r|r|}
\hline
$\kappa$ & uniform $p=3$ error & optimized error & $N_{\rm fev}$ & $\kappa_2(A)$ & $\kappa_{\GR}$\\
\hline
4  & $1.1875\times10^{-1}$ & $1.25\times10^{-15}$ & 6  & 106.69 & 7.49\\
8  & $3.7541\times10^{-1}$ & $1.19\times10^{-15}$ & 8  & 37.61  & 5.98\\
12 & $5.5874\times10^{-1}$ & $2.01\times10^{-15}$ & 12 & 40.35  & 5.47\\
\hline
\end{tabular}
\end{table}

The preceding calculation is an exact recovery problem, but the corresponding nonlinear minimization is not globally easy.  For a second hidden triple $(41^\circ,177^\circ,303^\circ)$, a direct $\kappa=8$ solve from a generic start converges to $(279.080^\circ,143.976^\circ,9.238^\circ)$ with $E_{L^2}=1.1513$ and $\mathcal J=50.43$.  Continuation from $\kappa=1$ to $8$ recovers the exact triple and gives $E_{L^2}=2.17\times10^{-15}$ in 21 cumulative residual evaluations.

The corresponding direct solve required 67 residual evaluations.  In this case continuation reaches a different local minimum and also uses fewer residual evaluations.  The comparison is only one initialization and does not give a success probability for the nonlinear solver.  It does, however, show why the analysis in Section~\ref{subsec:local-theory} is local and why continuation is used in the more difficult experiments below.

\subsection{Transmission at \texorpdfstring{$\omega=12$}{omega=12}}\label{subsec:transmission}
We next consider the transmission problem on $(-1,1)^2$.  The purpose of this example is to check the complex-angle parameterization in the two cases that occur at an interface: a propagating transmitted wave and an evanescent one.  We take $n_1=2$ below $y=0$ and $n_2=1$ above it, so that $\kappa_j=\omega n_j$ with $\omega=12$.  There are two plane-wave components in the lower material and one in the upper material.  The exact transmission solution is used only to prescribe the boundary data and to measure the error.  A direct solve at $\omega=12$ can converge to a false minimum, so we use the continuation sequence
\[
 1,2,4,6,8,10,12.
\]

For $69^\circ$ incidence the recovered directions are
\begin{equation}\label{eq:69angles}
 69^\circ,\qquad291^\circ,\qquad44.2143563895^\circ,
\end{equation}
with imaginary components at roundoff.  Tangential phase matching predicts
\begin{equation}\label{eq:Snell-validation}
 \theta_t=\arccos\!\left(\frac{\kappa_1\cos69^\circ}{\kappa_2}\right)
 =44.2143563895^\circ,
\end{equation}
which agrees with the recovered direction in the upper medium to the digits shown.

At $29^\circ$ incidence, total internal reflection gives
\begin{equation}\label{eq:29angles}
 z_1=29^\circ+O(10^{-16})i,\qquad z_2=331^\circ+O(10^{-16})i,
\end{equation}
and
\begin{equation}\label{eq:evangle}
 z_3=0+1.15828058554i.
\end{equation}
The physical decay parameter is
\begin{equation}\label{eq:eta-validation}
 \eta_{\rm tr}=\operatorname{arccosh}\!\left(\frac{\kappa_1\cos29^\circ}{\kappa_2}\right)
 =1.1582805855,
\end{equation}
The computed imaginary part agrees with \eqref{eq:eta-validation}.  Thus the same direction search finds the decaying branch without a critical-angle test or a separate evanescent ansatz.

The real parts of the two solutions are shown in Fig.~\ref{fig:transmission-fields}.  The mesh plays essentially no approximation role once the three directions have been recovered.  This is worth displaying explicitly.  Figure~\ref{fig:transmission-meshes} shows the 8- and 32-triangle meshes used in Table~\ref{tab:trans-mesh}.  Already on the 8-triangle mesh the relative $L^2$ errors are $1.41\times10^{-15}$ in the propagating case and $2.00\times10^{-15}$ in the evanescent case.  Refining to 32 triangles does not improve these numbers in any meaningful way; both calculations are already at working precision.  This is exactly what should happen: after the correct directions are found, the manufactured transmission field belongs to the piecewise Trefftz space on either mesh.

The conditioning tells a different story.  For the 8-triangle evanescent calculation, $\kappa_2(A)=1.22\times10^9$, whereas $\kappa_{\GR}=2.25$.  Thus the coarse mesh is already sufficient for approximation, but the choice of coefficient coordinates still has a large effect on the algebraic problem.
\begin{figure}[htbp]
\centering
\begin{minipage}[t]{0.43\textwidth}
\centering
\includegraphics[width=\textwidth]{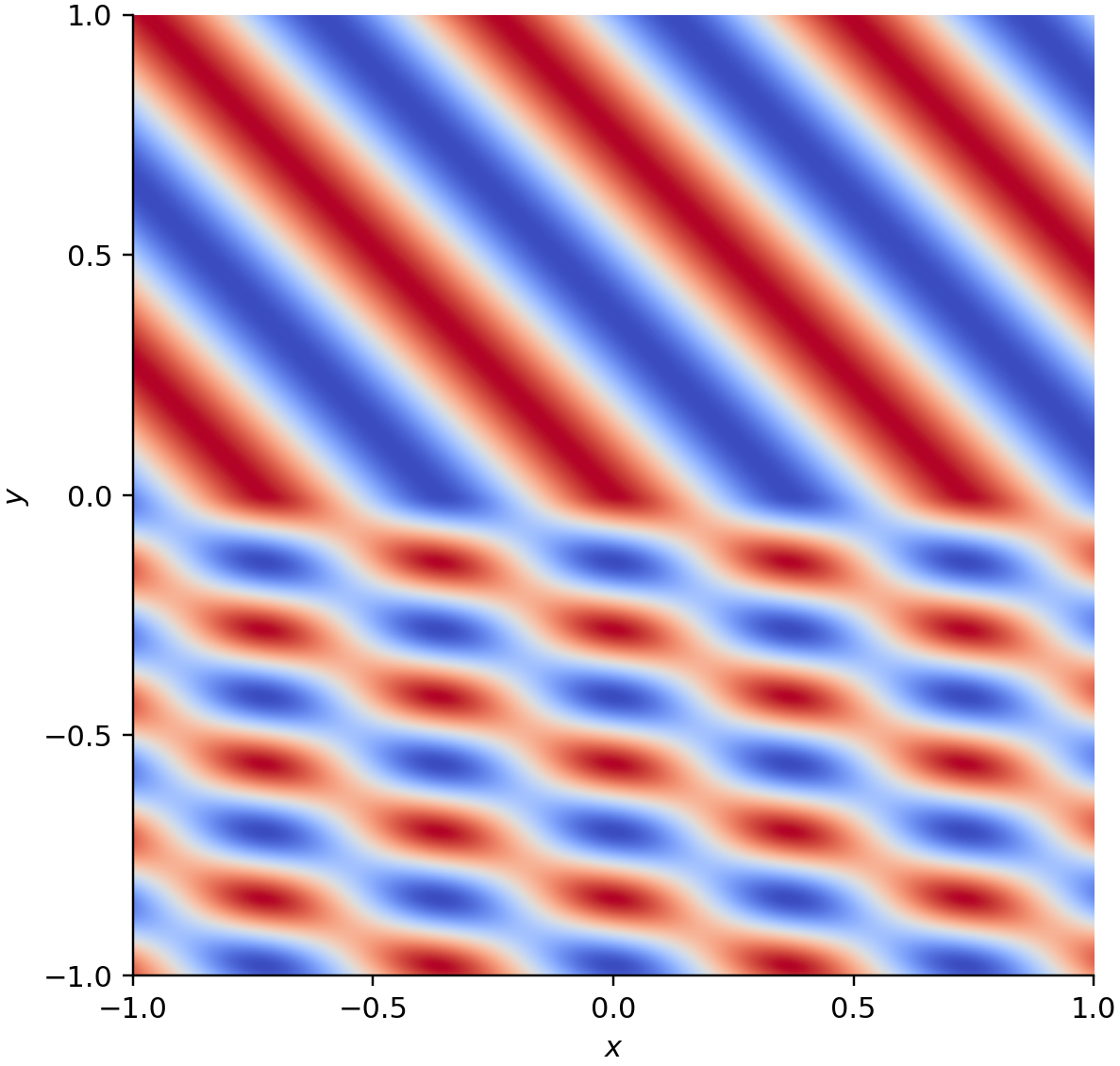}\\[-1.5mm]
(a) $69^\circ$ propagating transmission.
\end{minipage}\hfill
\begin{minipage}[t]{0.43\textwidth}
\centering
\includegraphics[width=\textwidth]{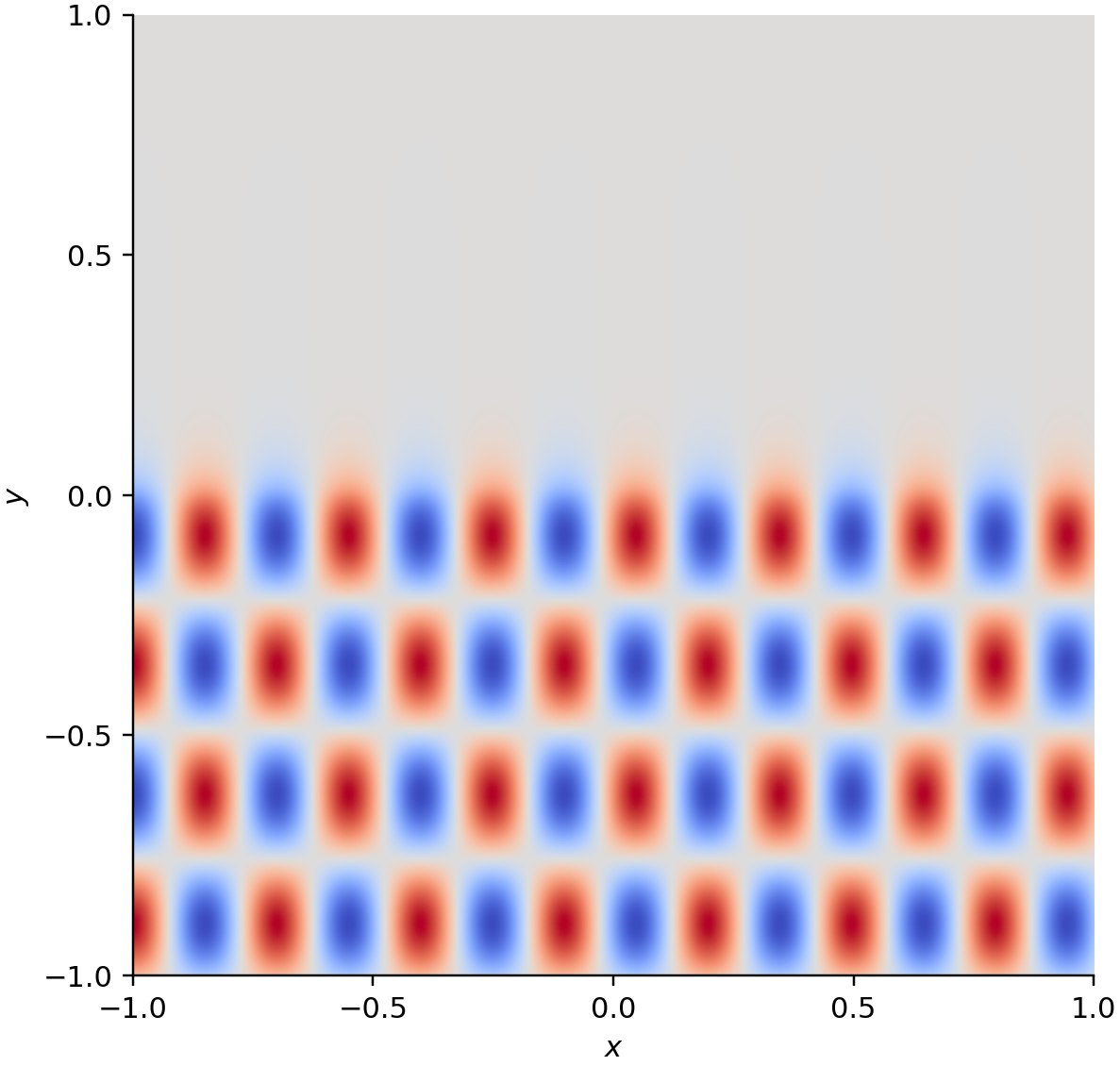}\\[-1.5mm]
(b) $29^\circ$ evanescent transmission.
\end{minipage}
\caption{Real parts of the $\omega=12$ transmission fields.}
\label{fig:transmission-fields}
\end{figure}
\begin{table}[htbp]
\centering
\caption{Part~A transmission recovery at $\omega=12$ on uniform square meshes.  $N_{\rm fev}$ is cumulative over continuation.}
\label{tab:trans-mesh}
\begin{tabular}{|c|c|r|r|r|r|r|}
\hline
triangles & incidence & coefficients & $N_{\rm fev}$ & $E_{L^2}$ & $\kappa_2(A)$ & $\kappa_{\GR}$\\
\hline
8 & $69^\circ$ & 12 & 15 & $1.41\times10^{-15}$ & 10.06 & 1.84\\
32 & $69^\circ$ & 48 & 15 & $1.13\times10^{-15}$ & 22.95 & 2.92\\
8 & $29^\circ$ & 12 & 21 & $2.00\times10^{-15}$ & $1.22\times10^9$ & 2.25\\
32 & $29^\circ$ & 48 & 20 & $4.26\times10^{-15}$ & $4.28\times10^4$ & 3.52\\
\hline
\end{tabular}
\end{table}

\begin{figure}[htbp]
\centering
\includegraphics[width=0.76\textwidth]{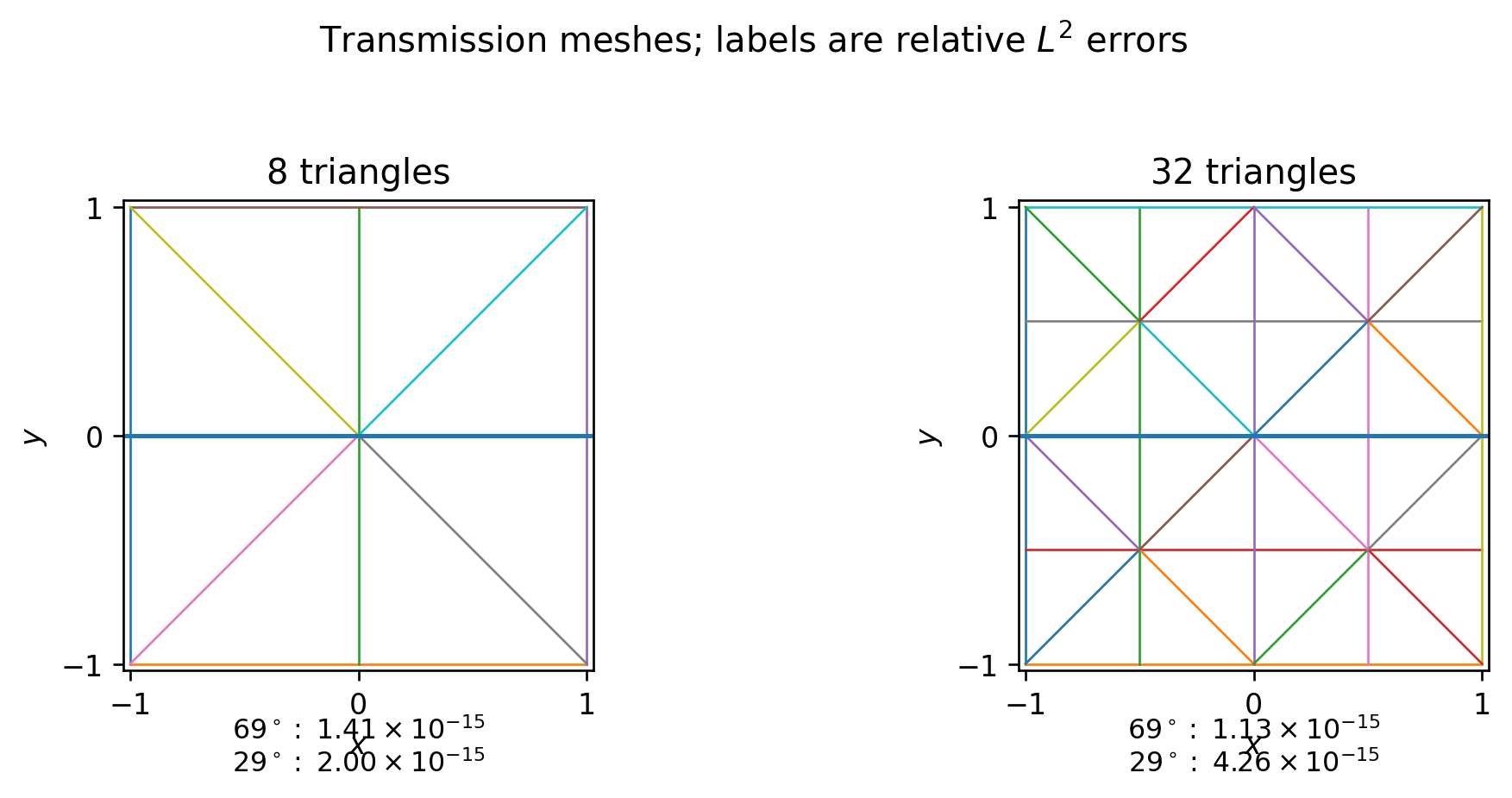}
\caption{The two uniform transmission meshes used in Table~\ref{tab:trans-mesh}.  The interface is $y=0$.  The numbers printed below each mesh are the relative $L^2$ errors for the $69^\circ$ and $29^\circ$ calculations.  The 8-triangle mesh is already at machine precision in both cases; refinement to 32 triangles has no observable approximation effect.}
\label{fig:transmission-meshes}
\end{figure}
From Table~\ref{tab:trans-mesh} we also observe that the raw condition number of the $29^\circ$ calculation decreases on the finer mesh.  This is consistent with the smaller elementwise exponential variation of the evanescent basis.  The graph--Riesz condition number remains small on both meshes.

\subsection{Exact circular DtN boundary: Hankel waves and the limitation of a sparse local space}\label{subsec:dtn-hankel}
We next return to the circular artificial boundary used in the DtN--PWDG experiments of \cite{KapitaMonk2018}.  We take a sound-soft circle of radius $a=0.5$ and a concentric artificial boundary $r=R$ with $R=1$.  As in the earlier work, both circular boundaries are represented by exact curved edges so that there is no polygonal geometry error.  With the time dependence $e^{+i\omega t}$, outgoing cylindrical modes are represented by Hankel functions of the second kind.

For $v(R,\theta)=\sum_{m\in\mathbb Z}v_m e^{im\theta}$, the truncated DtN map is
\begin{equation}\label{eq:dtn-map-new}
 S_N v=\sum_{|m|\le N} k\frac{H_m^{(2)\prime}(kR)}{H_m^{(2)}(kR)}v_m e^{im\theta},
 \qquad
 v_m=\frac{1}{2\pi}\int_0^{2\pi}v(R,\phi)e^{-im\phi}\,d\phi .
\end{equation}
On $\Gamma_R$ we use the DtN fluxes from \cite{KapitaMonk2018},
\begin{equation}\label{eq:dtn-flux-new}
\begin{aligned}
&\widehat u_h=u_h-\frac{\delta}{ik}\big(\partial_nu_h-S_Nu_h\big),\\
&ik\widehat\sigma_h=S_Nu_h\,n-\frac{\delta}{ik}S_N^\star\big(\nabla_hu_h-S_Nu_h\,n\big).
\end{aligned}
\end{equation}
For direction selection the boundary part of the skeleton objective is replaced by the nonlocal residual
\begin{equation}\label{eq:dtn-residual-new}
 \mathcal J_{R,N}(v_h)=\frac{\delta_D}{k}\|\partial_n v_h-S_Nv_h\|_{L^2(\Gamma_R)}^2,
\end{equation}
where $\delta_D$ denotes the optimization weight and is kept distinct from the flux parameter $\delta$.  The DtN term couples all outer boundary edges, so the corresponding part of $M(Z)$ is a global boundary block rather than an edge-local contribution.

We first use the DtN problem to expose a limitation of direction movement.  We ask whether a small local fan can follow the phase of an outgoing cylindrical field and, if it can, whether this is enough to give a small field error.  We consider the two exact solutions
\[
 u(x)=H_0^{(2)}(k|x-x_s|),\qquad k=8,
\]
with $x_s=0$ and with an off-center source $x_s$ lying inside the removed disk.  In both cases the exact field satisfies the homogeneous Helmholtz equation in the annulus and the exact DtN condition on $r=R$.  The arrows used below represent the learned plane-wave wave vectors $q$, equivalently the local spatial phase-gradient directions.  This distinction is important for the present time convention: since $e^{+i\omega t}H_0^{(2)}(kr)\sim e^{i\omega t-ikr}/\sqrt r$, an outgoing cylindrical wave has $q$ directed toward the source, whereas its physical energy propagates outward.  Thus the plotted arrows are not energy-flow arrows.  We use one radial layer, eight exact curved sectors, and a three-wave local fan with offsets $(-18^\circ,0,18^\circ)$; only the fan center is moved.  Frequency continuation uses $k=1,2,4,8$, and $N=16$ at the target frequency.

The same exact annular mesh is used throughout the DtN experiments.  For reference, Fig.~\ref{fig:dtn-meshes-used} shows this mesh together with the exact scattered field for the standard sound-soft disk problem with incident wave $u^{\rm inc}=e^{ikx}$.  With the present $e^{+i\omega t}$ convention, the scattered field is
\begin{equation}\label{eq:disk-scattering-series}
\begin{aligned}
u^{\rm scat}(r,\theta)
={}&-\frac{J_0(ka)}{H_0^{(2)}(ka)}H_0^{(2)}(kr)\\
&-2\sum_{m=1}^{\infty} i^m
\frac{J_m(ka)}{H_m^{(2)}(ka)}H_m^{(2)}(kr)\cos(m\theta).
\end{aligned}
\end{equation}
The field shown in Fig.~\ref{fig:dtn-meshes-used} uses $M=20$ terms, which is more than sufficient for the visualization at $k=8$.  This truncation is used only for the plot; the reported numerical errors use the converged reference series in the computational code.  The Hankel-source calculations below are retained as local phase-recovery diagnostics.  The high-$p$ experiment in Section~\ref{subsec:dtn-precision} returns to this plane-wave-incidence scattering problem.

\begin{figure}[htbp]
\centering
\includegraphics[width=0.98\textwidth]{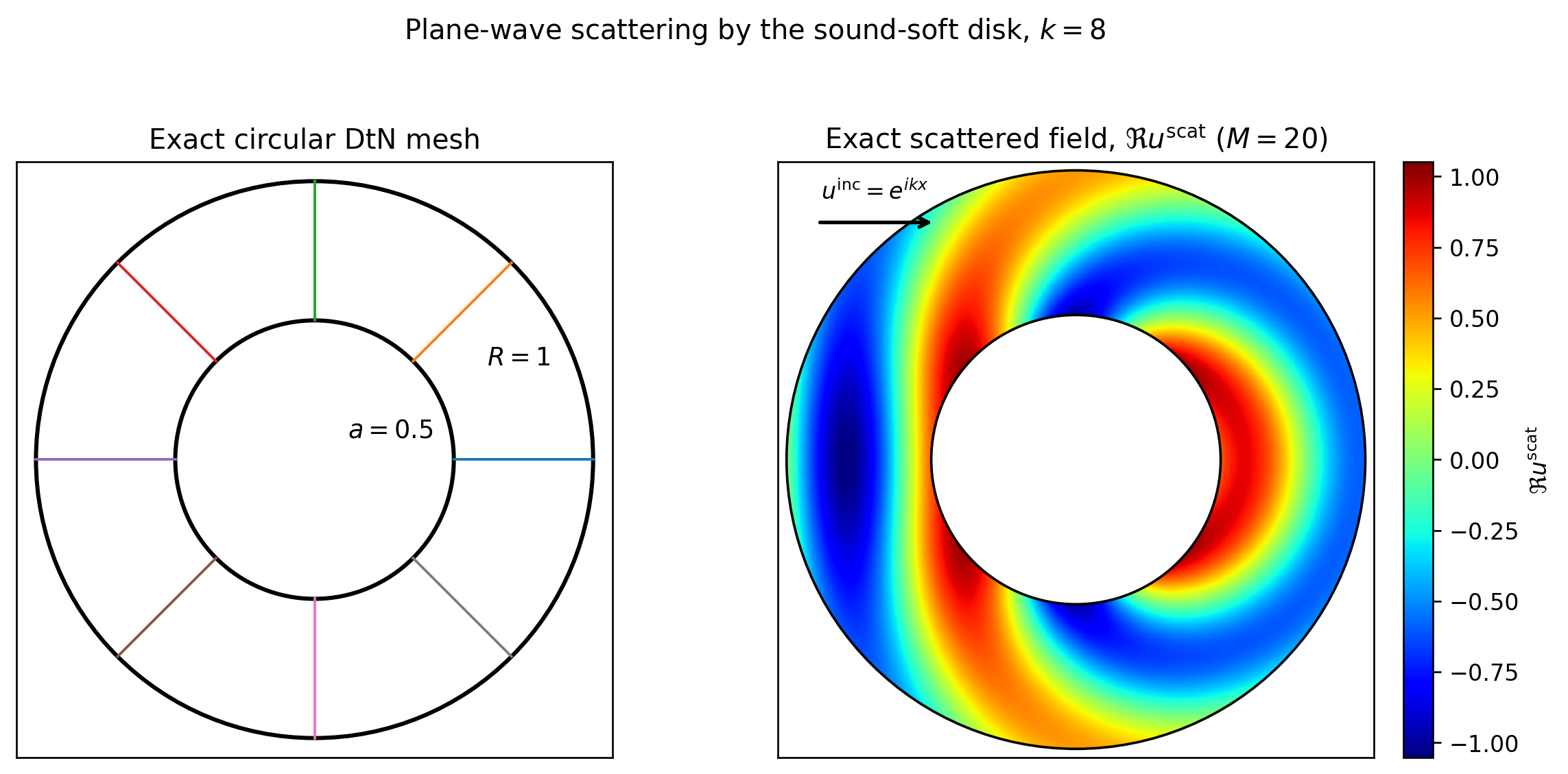}
\caption{DtN geometry and plane-wave scattering reference field at $k=8$.  Left: the exact annular mesh used throughout the DtN experiments, with one radial layer and eight exact curved sectors.  Right: the real part of the exact scattered field for the incident plane wave $u^{\rm inc}=e^{ikx}$, evaluated from \eqref{eq:disk-scattering-series} with $M=20$ terms for display.  The $M=20$ truncation is used only for this plot and does not alter any of the numerical results reported below.}
\label{fig:dtn-meshes-used}
\end{figure}

\begin{table}[htbp]
\centering
\caption{DtN Hankel tests on the exact annulus.  The angular errors compare the learned outer-ring fan center with the exact local wave-vector direction and are used only as a posteriori diagnostics.}
\label{tab:dtn-hankel}
\begin{tabular}{|l|r|r|r|r|r|}
\hline
case & uniform $p=3$ $E_{L^2}$ & adaptive $E_{L^2}$ & mean angle error & max. angle error & $\kappa_{\GR}$\\
\hline
centered source & $7.653\times10^{-1}$ & $7.663\times10^{-2}$ & $7\times10^{-6}\!{}^\circ$ & $7\times10^{-6}\!{}^\circ$ & 1.307\\
off-center source & $7.626\times10^{-1}$ & $9.407\times10^{-2}$ & $2.86^\circ$ & $4.87^\circ$ & 1.339\\
\hline
\end{tabular}
\end{table}

\begin{figure}[htbp]
\centering
\begin{minipage}[t]{0.48\textwidth}
\centering
\includegraphics[width=\textwidth]{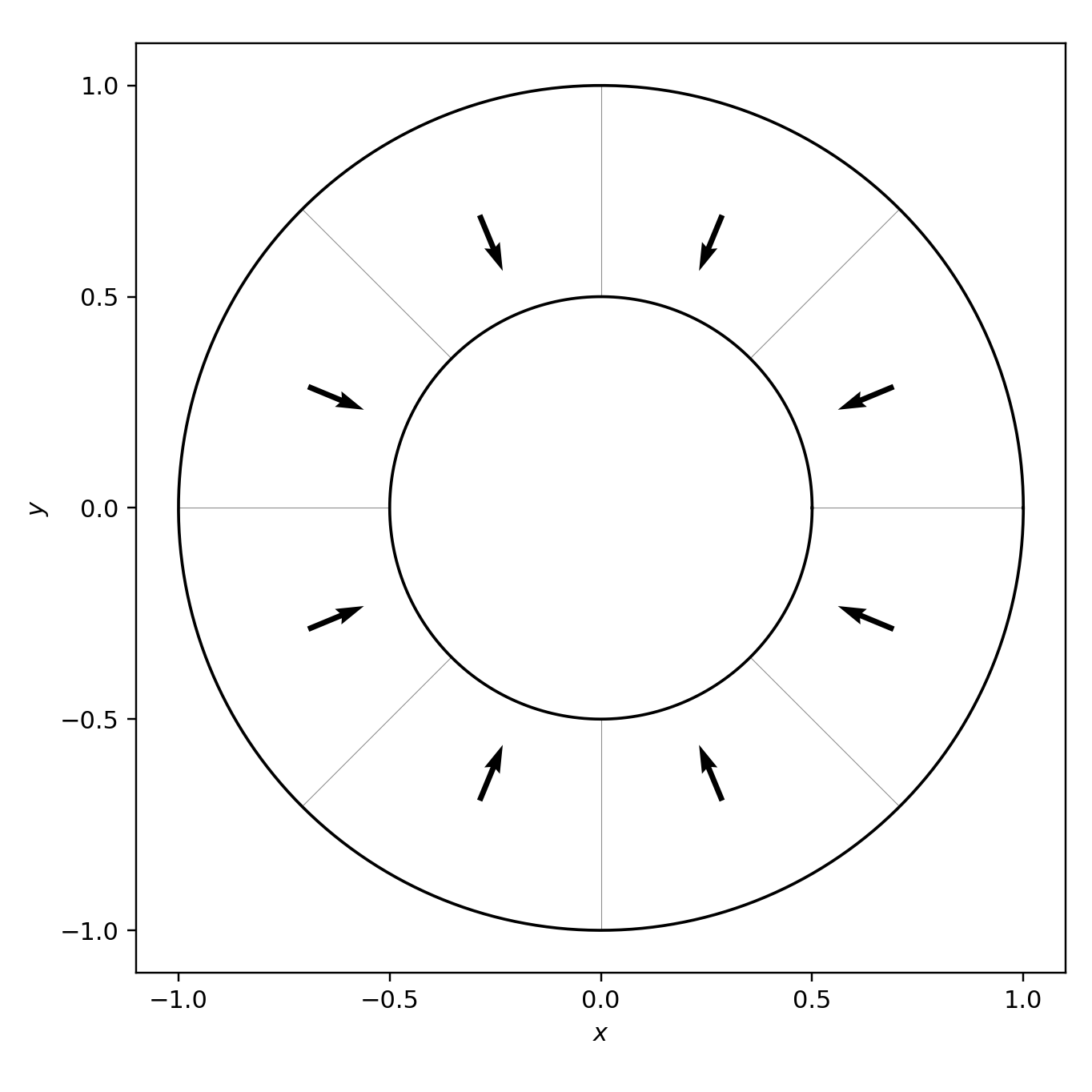}\\[-1.5mm]
(a) Centered cylindrical wave.
\end{minipage}\hfill
\begin{minipage}[t]{0.48\textwidth}
\centering
\includegraphics[width=\textwidth]{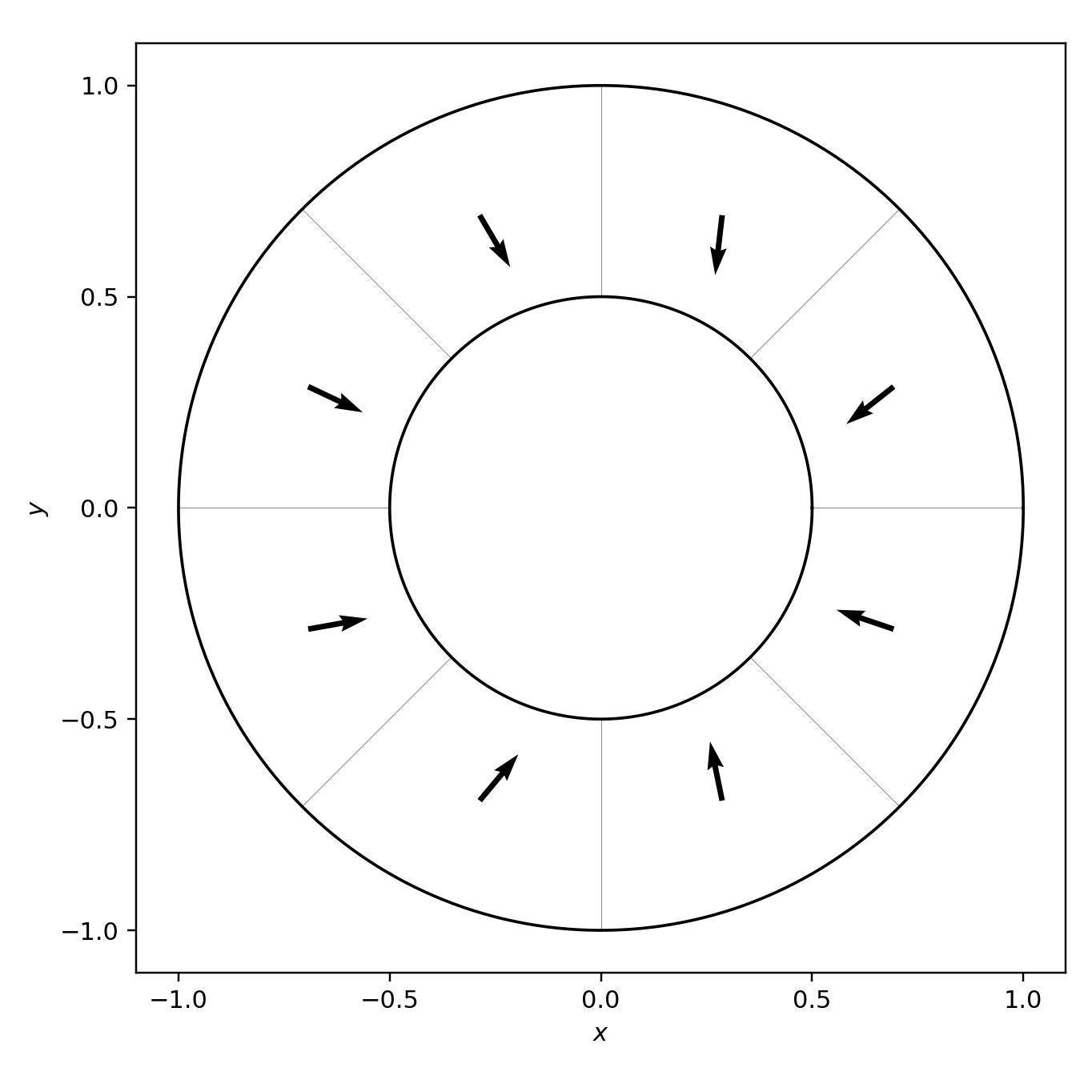}\\[-1.5mm]
(b) Off-center cylindrical wave.
\end{minipage}
\caption{Learned plane-wave wave vectors on the exact circular DtN boundary.  The arrows show $q$, equivalently the local spatial phase-gradient direction, not the physical energy-propagation direction.  With the $e^{+i\omega t}$ convention and outgoing $H^{(2)}$ waves, $q$ points toward the source while energy propagates outward.  Thus, in the centered case the recovered arrows point radially inward to numerical accuracy; in the off-center case they follow the source-centered phase direction rather than the normal of the artificial circle.}
\label{fig:dtn-hankel-directions}
\end{figure}

From Table~\ref{tab:dtn-hankel} and Fig.~\ref{fig:dtn-hankel-directions} we observe that the direction recovery itself is good.  For the centered source the fan centers are radially inward to the digits shown.  For the off-center source they follow the source-centered phase rather than the normal to the artificial circle.  Nevertheless, the relative $L^2$ error remains of order $10^{-1}$.  The conclusion is simple: recovering the dominant local phase direction is not enough for a cylindrical field.  A three-wave fan does not represent the curvature of the wavefront or the radial amplitude variation.  We next increase the local dimension and ask whether the remaining high-$p$ floor is approximation error or numerical error.

\subsection{Finite-precision, compression and trace-rank barriers for DtN--PWDG}\label{subsec:dtn-precision}
Our next DtN experiment examines the high-$p$ calculation.  In \cite{KapitaMonk2018} the relative error decreased rapidly with $p$ and then stopped decreasing because of numerical instability.  Here we investigate the source of this behavior.  We use the same exact circular geometry, but we do not try to reproduce the earlier mesh point for point.  Instead, we fix the present discretization and compare the raw coefficient system, the graph--Riesz system, different trace cutoffs, and different arithmetic precisions.  In this way we can distinguish approximation error from numerical effects.

We now return to the sound-soft disk scattering problem shown in Fig.~\ref{fig:dtn-meshes-used}, with incident field $u^{\rm inc}=e^{ikx}$ and $u^{\rm scat}=-u^{\rm inc}$ on $r=a$.  We take the exact annulus $0.5<r<1$, $k=8$, one radial layer and eight exact curved sectors.  We fix the DtN truncation at $N=40$ and use 44-point Gauss--Legendre quadrature on the circular and radial edges.  The local directions are uniformly distributed and we use
\[
 p=19,23,25,27,29,31,33,35,37,39.
\]
The same nominal spaces are solved first in the raw coefficient basis in IEEE double precision.  We then apply the local Cauchy-trace normalization of Section~\ref{subsec:compression} and solve in graph--Riesz coordinates, again in double precision.  A companion calculation uses 80-bit extended arithmetic.  For $p=25,27,29,31,33$ we additionally assemble and factor the full system in IEEE binary128 arithmetic.  No high-precision values are interpolated or extrapolated.

Results are shown in Fig.~\ref{fig:dtn-all-methods-floor}, with selected values in Table~\ref{tab:dtn-precision}.  Up to $p=27$ the different calculations give essentially the same error.  At $p=29$ the raw double-precision error increases instead of decreasing.  The graph--Riesz calculation does not show this large increase, which indicates that the first loss of accuracy is caused by the coefficient basis.  Increasing the arithmetic precision gives a further improvement.  For example, at $p=33$ the binary128 error is $2.56\times10^{-5}$, compared with approximately $2.3\times10^{-4}$ for graph--Riesz in double precision.  Thus graph--Riesz normalization improves the double-precision calculation, but it does not by itself recover the smaller modes that are visible at higher precision.

\begin{table}[htbp]
\centering
\caption{Selected points from the audited DtN precision continuation.  The graph--Riesz column uses double precision and a near-machine trace cutoff.}
\label{tab:dtn-precision}
\begin{tabular}{|r|r|r|r|r|}
\hline
$p$ & raw double $E_{L^2}$ & GR double $E_{L^2}$ & 80-bit $E_{L^2}$ & binary128 $E_{L^2}$\\
\hline
25 & $4.506\times10^{-4}$ & $4.507\times10^{-4}$ & $4.507\times10^{-4}$ & $4.507\times10^{-4}$\\
27 & $2.417\times10^{-4}$ & $2.394\times10^{-4}$ & $2.394\times10^{-4}$ & $2.394\times10^{-4}$\\
29 & $1.191\times10^{-3}$ & $2.347\times10^{-4}$ & $1.065\times10^{-4}$ & $1.063\times10^{-4}$\\
31 & $3.340\times10^{-4}$ & $2.358\times10^{-4}$ & $8.302\times10^{-5}$ & $5.420\times10^{-5}$\\
33 & $3.087\times10^{-4}$ & $2.341\times10^{-4}$ & $9.274\times10^{-5}$ & $2.555\times10^{-5}$\\
\hline
\end{tabular}
\end{table}

\begin{figure}[htbp]
\centering
\includegraphics[width=0.80\textwidth]{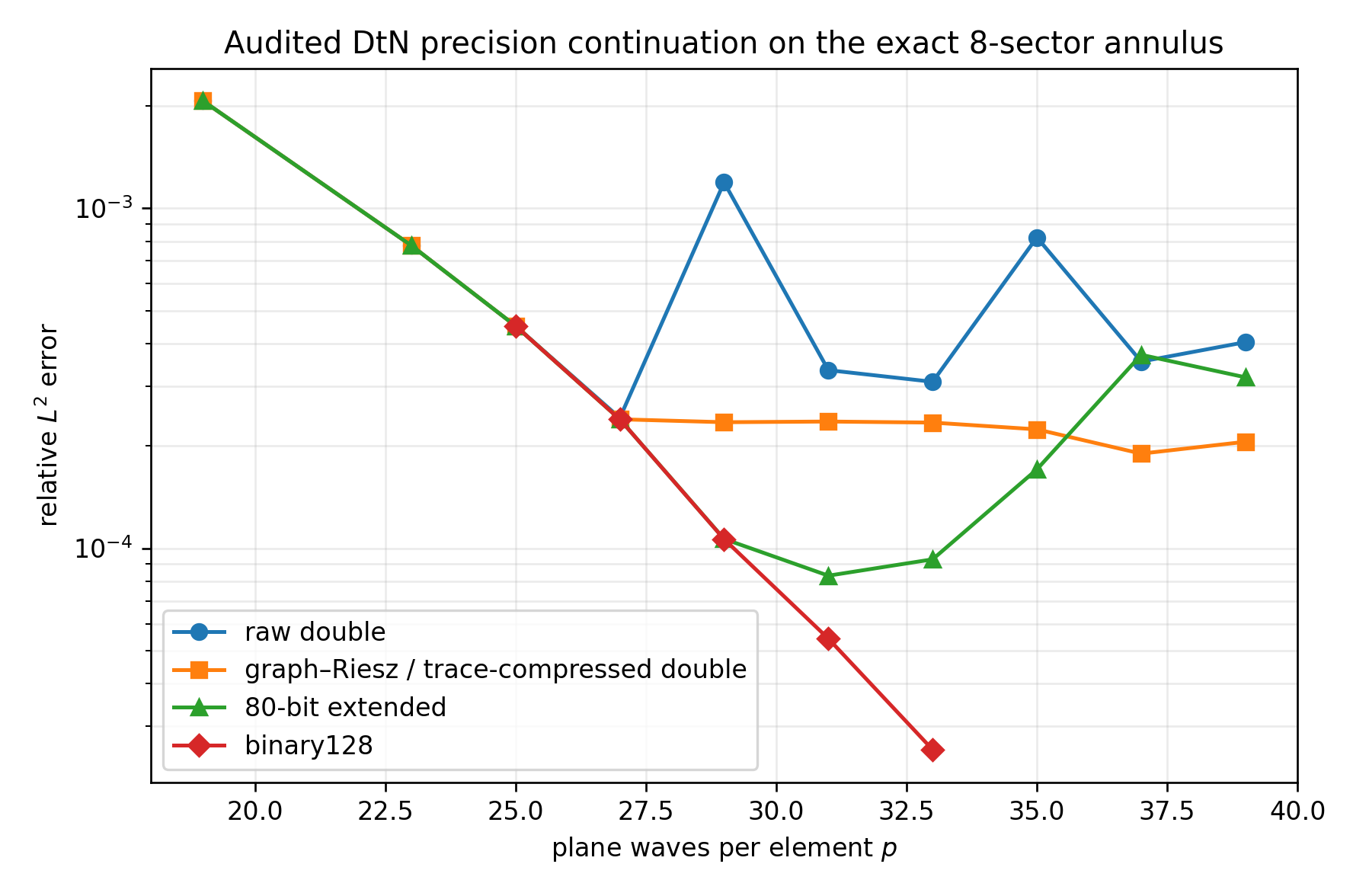}
\caption{Audited precision continuation on the exact eight-sector DtN annulus.  Raw double, graph--Riesz double and 80-bit arithmetic contain ten computed points.  Binary128 assembly and factorization were carried out only at $p=25,27,29,31,33$.  The raw double basis fails first; graph--Riesz normalization stabilizes the trace space still visible in double precision; higher arithmetic resolves additional modes and moves the $p$-floor.}
\label{fig:dtn-all-methods-floor}
\end{figure}

We next investigate the effect of the trace cutoff.  We keep the trace eigensolve in double precision and decrease $\tau_{\rm rank}$.  Results for three values of $p$ are given in Table~\ref{tab:dtn-tau}.  From the table we see that decreasing the cutoff continues to add useful modes.  For $p=35$, the relative error decreases from $4.27\times10^{-4}$ when $\tau_{\rm rank}=10^{-14}$ to $6.66\times10^{-5}$ when all trace modes having a computed positive eigenvalue are retained.  Hence the choice $10^{-12}$ is too large for this high-$p$, high-accuracy calculation.

\begin{table}[htbp]
\centering
\caption{Effect of decreasing the trace cutoff in the double-precision graph--Riesz calculation.  ``all $+$'' means that every trace eigenvector with a computed positive eigenvalue is retained.}
\label{tab:dtn-tau}
\begin{tabular}{|r|c|r|r|}
\hline
$p$ & $\tau_{\rm rank}$ & effective DOFs & $E_{L^2}$\\
\hline
29 & $10^{-12}$ & 184 & $7.90\times10^{-4}$\\
29 & $10^{-15}$ & 216 & $2.36\times10^{-4}$\\
29 & $10^{-17}$ & 224 & $1.79\times10^{-4}$\\
29 & all $+$ & 225 & $1.70\times10^{-4}$\\
\hline
31 & $10^{-15}$ & 216 & $2.35\times10^{-4}$\\
31 & $10^{-16}$ & 220 & $1.94\times10^{-4}$\\
31 & $10^{-17}$ & 233 & $1.58\times10^{-4}$\\
31 & all $+$ & 236 & $1.33\times10^{-4}$\\
\hline
35 & $10^{-14}$ & 200 & $4.27\times10^{-4}$\\
35 & $10^{-16}$ & 223 & $1.75\times10^{-4}$\\
35 & $10^{-17}$ & 248 & $8.13\times10^{-5}$\\
35 & all $+$ & 249 & $6.66\times10^{-5}$\\
\hline
\end{tabular}
\end{table}

Decreasing the cutoff in double precision eventually stops being sufficient because the smallest eigenvalues of the trace Gramian are no longer computed reliably.  To check this directly, we recomputed the local trace Gramian and its eigendecomposition at 50 decimal digits for $p=29,31,35$.  In all three cases every eigenvalue is positive.  The smallest relative eigenvalues are
\[
 5.01\times10^{-13},\qquad 1.00\times10^{-14},\qquad 2.53\times10^{-18}.
\]
Consequently, at $p=35$ the fixed value $\tau_{\rm rank}=10^{-12}$ retains only 27 of 35 local trace modes, $10^{-16}$ retains 33, and $10^{-20}$ retains all 35.  Figure~\ref{fig:dtn-trace-rank} shows this directly.

\begin{figure}[htbp]
\centering
\includegraphics[width=0.72\textwidth]{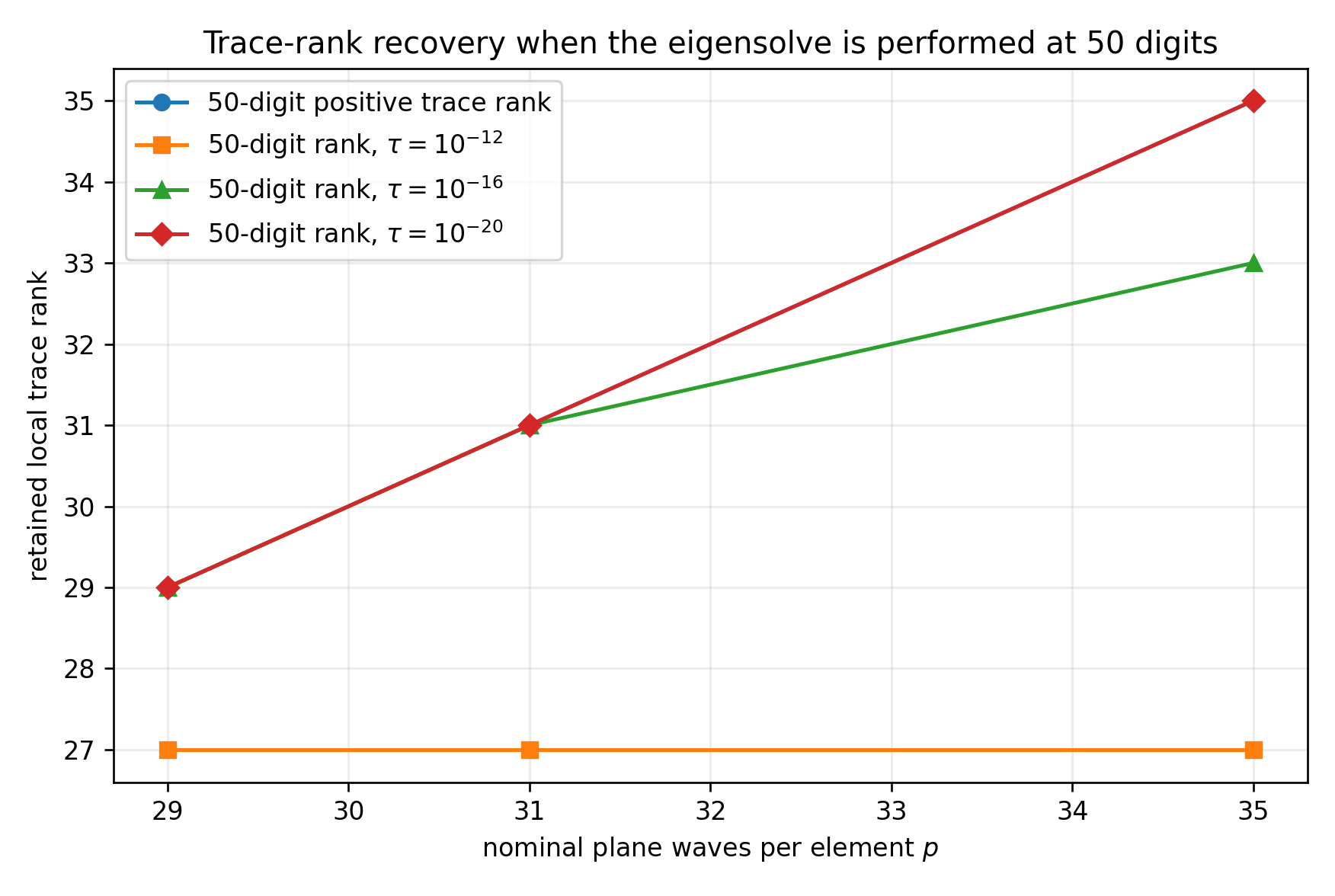}
\caption{Local trace rank when the Gramian and its eigendecomposition are computed at 50 decimal digits.  The high-precision spectrum is positive in all three cases.  A fixed cutoff $10^{-12}$ therefore creates an artificial high-$p$ rank ceiling; decreasing the cutoff in proportion to the working precision restores the smaller but still resolvable trace modes.}
\label{fig:dtn-trace-rank}
\end{figure}

From these experiments we draw two conclusions.  First, $\tau_{\rm rank}$ should not be regarded as a fixed universal value.  It must be chosen with the working precision and the accuracy required from the calculation.  Second, graph--Riesz normalization and trace truncation address different numerical problems.  Graph--Riesz normalization improves the conditioning on the retained space, while the cutoff decides which trace modes are retained.  Once the trace spectrum itself reaches the double-precision noise level, the eigensolve must also be carried out at higher precision.  We do not pursue a full arbitrary-precision compressed global solve here.

\subsection{Finite directional complexity and the limit of exact recovery}\label{subsec:finite-capacity}
The Hankel calculation has broad directional content, so it does not isolate the ability of the nonlinear search to recover a finite set of rays.  We now remove this ambiguity.  We prescribe the exact solution to be a finite sum of plane waves and increase the number of directions one at a time.  The exact field therefore remains in a finite Trefftz space for every value of $M$.  Any loss of exact recovery can then be attributed to the direction search or the numerical linear algebra, not to curved-wavefront approximation.  On $\Omega=(-1,1)^2$ with $\kappa=12$ we take
\begin{equation}\label{eq:finite-sum}
 u_M(x)=\sum_{j=1}^{M}a_j\exp(i\kappa d_j\cdot x),
 \qquad d_j=(\cos\theta_j,\sin\theta_j),
\end{equation}
with $|a_j|=1$.  The nested direction sequence is
\[
\begin{aligned}
 &(\theta_j)_{j=1}^{20}=(17,123,251,68,191,315,39,158,286,340,\\
 &\hspace{8.8em}92,224,5,145,273,327,54,178,235,301)^\circ,
\end{aligned}
\]
and the coefficient phases $a_j=e^{i\phi_j}$ are
\[
\begin{aligned}
 &(\phi_j)_{j=1}^{20}=(0,0.37,-0.51,0.81,-1.04,1.33,-1.52,0.22,1.71,-0.93,\\
 &\hspace{8.8em}0.58,-1.22,1.48,-0.31,0.96,-1.73,0.11,1.18,-0.77,1.91).
\end{aligned}
\]
The phases are fixed only to avoid an artificially symmetric sum.  Exact Dirichlet data are imposed on the boundary of a uniform 8-triangle mesh.

For each $M$ we use $p=M$ directions shared by all elements.  Hence the adaptive and uniform calculations have the same $8M$ coefficients.  The uniform directions are $11^\circ+360^\circ j/M$.  For the adaptive calculation we use continuation in $M$.  After the $M-1$ problem has converged, the new direction is chosen by testing a fixed angular dictionary with the skeleton residual, subject to a minimum separation from the directions already present.  We then move all $M$ directions by nonlinear least squares.  When ten directions are active we refine the birth dictionary from 36 to 72 angles; this only refines the search grid and does not prescribe the number of plane waves.  The exact directions in \eqref{eq:finite-sum} are not used by either ENRICH or MOVE.  They are used only afterward to measure the angular error.  All computations in this subsection are in IEEE double precision, including the $M=20$ calculation and the control experiment below.

The result is shown in Table~\ref{tab:finite-capacity} and Fig.~\ref{fig:finite-capacity}.  For $M=1,\ldots,19$ all directions are recovered to essentially machine precision and the relative $L^2$ error stays between $10^{-15}$ and $2\times10^{-14}$.  The comparison with the uniform space is the important point: both calculations use exactly $8M$ coefficients.  At $M=10$ the uniform error is $0.830$, at $M=15$ it is $0.539$, and at $M=19$ it is still $2.09\times10^{-2}$.  Thus the large gain is due to learning the correct directions, not to increasing the dimension of the approximation space.

\begin{table}[htbp]
\centering
\caption{Finite-direction capacity experiment at $\kappa=12$ on 8 triangles.  Both methods use $p=M$ directions and $8M$ coefficients.  The angle error is the maximum circular matching error between the recovered and exact direction sets.}
\label{tab:finite-capacity}
\begin{tabular}{|r|r|r|r|r|}
\hline
$M$ & adaptive $E_{L^2}$ & uniform $E_{L^2}$ & max. angle error & $\kappa_{\GR}$\\
\hline
10 & $2.91\times10^{-14}$ & $8.30\times10^{-1}$ & $7.38\times10^{-13}\!{}^\circ$ & --\\
15 & $4.28\times10^{-15}$ & $5.39\times10^{-1}$ & $1.97\times10^{-13}\!{}^\circ$ & 8.73\\
17 & $6.98\times10^{-15}$ & $1.16\times10^{-1}$ & $1.74\times10^{-12}\!{}^\circ$ & 16.08\\
18 & $2.16\times10^{-14}$ & $8.92\times10^{-2}$ & $2.56\times10^{-12}\!{}^\circ$ & 18.01\\
19 & $1.99\times10^{-14}$ & $2.09\times10^{-2}$ & $3.66\times10^{-12}\!{}^\circ$ & 32.40\\
20 & $1.62\times10^{-3}$  & $7.86\times10^{-3}$ & $28.99^\circ$ & 75.25\\
\hline
\end{tabular}
\end{table}

The final point, $M=20$, behaves differently.  This value is not an optimal dimension; it is simply the first value in this continuation for which the automatic birth step fails.  The residual test chooses $117.5^\circ$ for the new ray, whereas the twentieth exact direction is $301^\circ$.  MOVE decreases the residual but remains in a false basin.  The maximum angular error is then $28.99^\circ$ and $E_{L^2}=1.62\times10^{-3}$.  The uniform calculation at the same dimension has error $7.86\times10^{-3}$.  Thus the adaptive calculation is still better, but the large advantage seen for $M\le19$ is lost.  We also checked whether this is a double-precision linear algebra effect.  The final PWDG solve has relative algebraic residual $2.49\times10^{-14}$, with $\kappa_2(A)=1.65\times10^7$ and $\kappa_{\GR}=75.25$.  Since the field error is much larger than the algebraic residual, the loss of accuracy is caused by the nonlinear direction search rather than by the linear solve.

We next repeat only the twentieth step in order to check that the twenty-ray space can still represent the exact solution.  The first nineteen recovered directions are kept and the new ray is initialized at $302.5^\circ$, which is $1.5^\circ$ from the exact direction.  This calculation is again carried out in double precision.  After nine residual evaluations we obtain $E_{L^2}=1.11\times10^{-14}$, $\mathcal J=1.57\times10^{-24}$ and a maximum angular error of $3.20\times10^{-11}$ degrees.  Here $\kappa_2(A)=1.60\times10^4$ and $\kappa_{\GR}=35.54$.  This control shows that the twenty-ray Trefftz space has not reached an approximation or precision barrier.  The failure of the automatic calculation is caused by the birth decision.

\begin{figure}[htbp]
\centering
\includegraphics[width=0.80\textwidth]{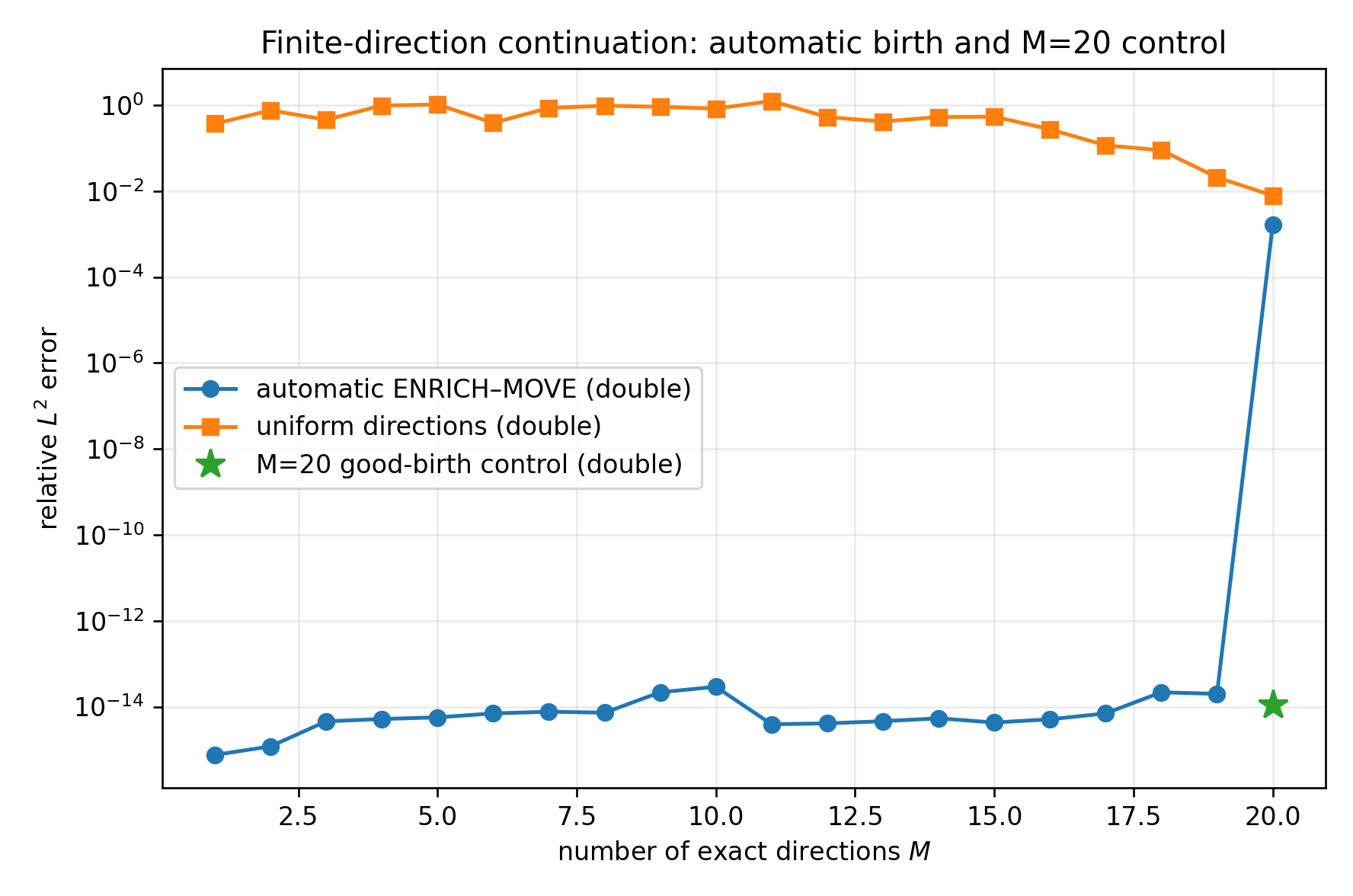}
\caption{Continuation in the number of exact directions.  Every point in this figure is computed in IEEE double precision.  The automatic ENRICH--MOVE path remains at roundoff through $M=19$ and enters a false basin at $M=20$.  The isolated star at $M=20$ is the double-precision control in which the twentieth ray is born at $302.5^\circ$; the same twenty-ray space then returns to roundoff.  This separates failure of the automatic birth decision from a finite-precision or representability barrier.}
\label{fig:finite-capacity}
\end{figure}

From Fig.~\ref{fig:finite-capacity} we see a clear change of behavior.  When the solution contains a finite set of identifiable directions, the adapted space can be much more accurate than a uniform angular space of the same dimension.  As more directions are added, the uniform space becomes a better generic approximation while the nonlinear birth problem becomes more difficult.  We do not attach special significance to the number nineteen.  It is the last exact-recovery point for this nested family and this continuation strategy.  The main observation is that the largest gain is obtained when the directional content is sparse enough to be identified.

\subsection{Trianglewise direction and \texorpdfstring{$h$}{h}-adaptivity}\label{subsec:hybrid-adapt}
Our final adaptive experiment combines direction changes with mesh refinement on the L-shaped corner-singular and crossing-wave example.  We compare the number of Trefftz coefficients required by the combined method with residual based $h$-refinement alone.  We measure the error by the relative graph norm
\[
 E_{\mathcal G}=
 \frac{\|u-u_h\|_{\mathcal G,h}}{\|u\|_{\mathcal G,h}}.
\]
Both methods start from the same 24-triangle mesh with five equally spaced plane waves per element.  The standard loop and its extension here are
\[
\begin{array}{c}
\text{SOLVE}\to\text{ESTIMATE}\to\text{MARK}\to\text{REFINE},\\[-1mm]
\text{SOLVE}\to\text{ESTIMATE}\to\text{MARK}\to\text{TEST}\to\text{ACT}\to\text{SOLVE}.
\end{array}
\]
In the ESTIMATE step we compute the element quantities $\eta_K^2$ from \eqref{eq:edge-indicator}.  Each interior edge contribution is divided equally between the two adjacent triangles, while a boundary contribution is assigned to the adjacent triangle.  This step uses only the current PWDG solution.  No trial direction or trial mesh problem is solved during ESTIMATE.

For the present all-Dirichlet problem these indicators have an exact interpretation.  Proposition~\ref{prop:exactJ} gives
\[
 \sum_K\eta_K^2=\mathcal J(u_h)=\|u-u_h\|_{\mathcal G,h}^2,
\]
so their sum is the squared graph error of the current PWDG solution.  For a mixed Dirichlet--impedance problem we would instead use the reliability estimate from Section~\ref{subsec:mixed-control}.  The role of ESTIMATE is therefore the same as in a standard adaptive method: it identifies the elements carrying the largest residual contribution.  Only after MARK has selected these elements do we use TEST to decide whether to move directions, add directions, or refine the mesh.

For consistency with the finite-direction study, we use the common nominal cap $p_K\le20$ during the direction-learning stage.  After cycle 14 the local fans are frozen and the remaining iterations use residual based $h$-refinement only.  The pure $h$ calculation starts from the same mesh and uses the same D\"orfler parameter $\theta=0.4$ and the same refinement rule, but keeps five fixed equally spaced directions on every element.

In the first stage, MARK applies D\"orfler marking \cite{Dorfler1996} to the values $\eta_K^2$.  For each marked element, TEST considers two MOVE trials, obtained by rotating the current fan by $\pm15^\circ$, and at most six ENRICH trials.  Candidate directions are required to be separated by at least $7^\circ$.  If $J_K^{\rm move}$ and $J_K^{\rm enrich}$ are the best trial residuals and $\eta_K^2$ the current local contribution, the action scores are
\begin{equation}\label{eq:hybrid-gains}
 G_K^{\rm move}=\frac{J-J_K^{\rm move}}{\eta_K^2},
 \qquad
 G_K^{\rm enrich}=\frac{J-J_K^{\rm enrich}}{\eta_K^2}.
\end{equation}
ACT chooses the stronger admissible directional action when
\[
 G_K^{\rm move}\ge0.20
 \qquad\hbox{or}\qquad
 G_K^{\rm enrich}\ge0.40,
\]
and otherwise chooses REFINE.  Children inherit the parent direction set.

The main stages of the calculation are listed in Table~\ref{tab:hybrid-threshold}.  At the end of the direction-learning stage the relative graph error has decreased from $0.8471$ to $0.1278$ using 884 coefficients.  The mean local dimension is $6.23$ and the largest local dimension is 8, so the common cap $p_K\le20$ is inactive and does not affect this trajectory.  With the learned fans then held fixed, the hybrid calculation first satisfies $E_{\mathcal G}<0.05$ with 3329 coefficients.  The pure $h$ calculation requires 11290 coefficients to reach the same error level.

\begin{table}[htbp]
\centering
\caption{Threshold-driven hybrid benchmark on the L-shaped domain.}
\label{tab:hybrid-threshold}
\begin{tabular}{|l|r|r|r|r|}
\hline
state & cycle & triangles & coefficients & $E_{\mathcal G}$\\
\hline
initial mesh & 0 & 24 & 120 & 0.8471\\
hybrid after direction stage & 14 & 142 & 884 & 0.1278\\
hybrid at threshold & 19 & 538 & 3329 & 0.0404\\
pure $h$ at threshold & 16 & 2258 & 11290 & 0.0462\\
\hline
\end{tabular}
\end{table}

\begin{figure}[htbp]
\centering
\includegraphics[width=0.68\textwidth]{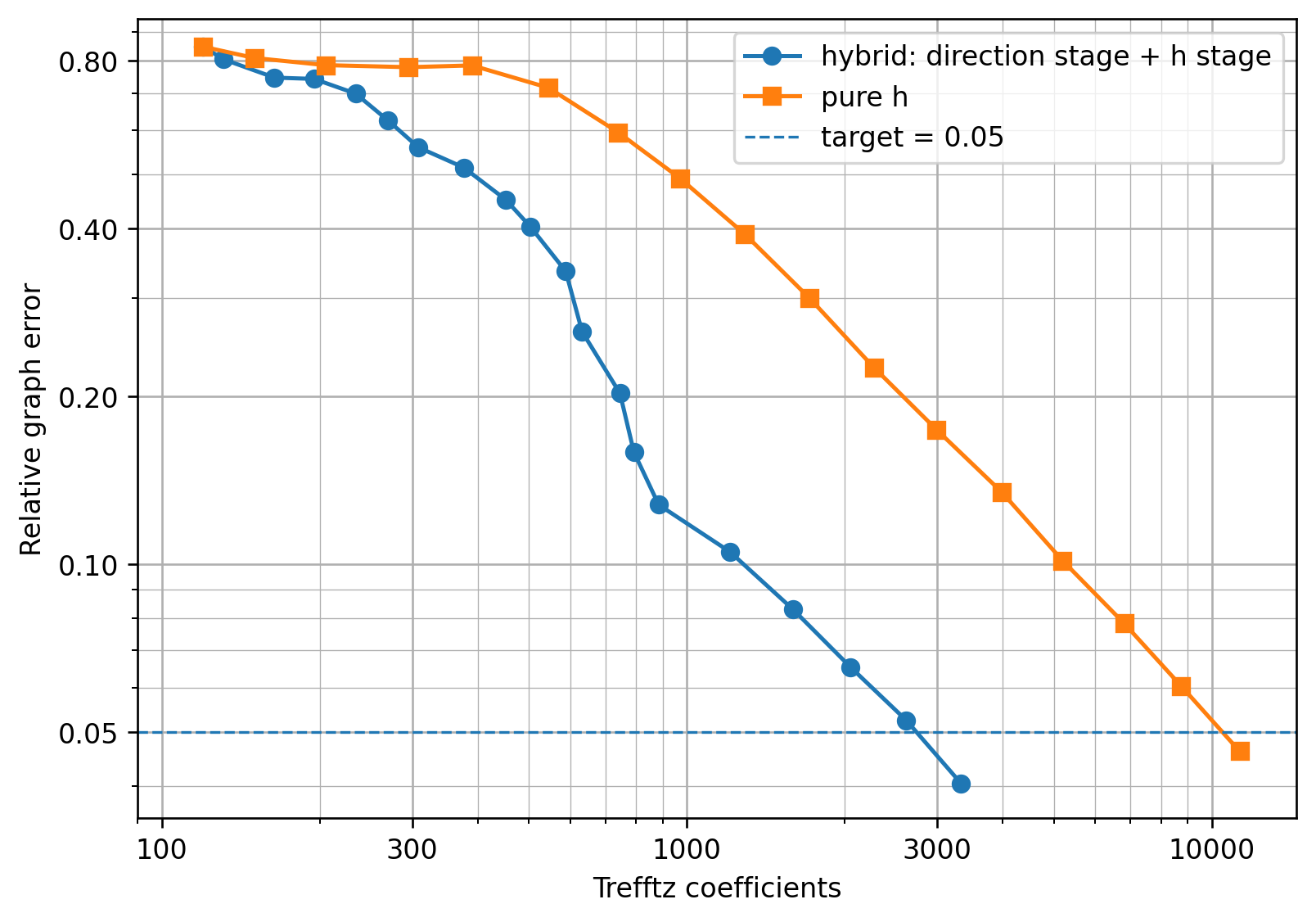}
\caption{Relative graph error versus Trefftz coefficients for the threshold-driven hybrid benchmark.  The hybrid method uses a direction-learning stage followed by pure $h$-refinement with the learned local fans frozen.}
\label{fig:hybrid-threshold}
\end{figure}

The complete histories are shown in Fig.~\ref{fig:hybrid-threshold}.  At the first point below $5\%$ relative graph error, the hybrid calculation uses about $70.5\%$ fewer coefficients than pure $h$-refinement.  During the first stage, MOVE and ENRICH are selected where changing the local directional space gives the larger reduction in the residual, while the mesh is still refined near the singular corner.  The learned directions are then frozen.  From the last part of the figure we see that the reduction in coefficients is retained as the mesh is further refined.

\paragraph{TEST-stage work.}
The present TEST implementation is not cheaper than pure $h$-adaptivity in elapsed time.  Every marked element may require two MOVE trials and as many as six ENRICH trials, and each accepted trial is evaluated through a global solve.  Table~\ref{tab:hybrid-threshold} and Fig.~\ref{fig:hybrid-threshold} should therefore be interpreted as a comparison of approximation spaces, not as a timing comparison.

\subsection{Variable-projection check}\label{subsec:joint-results}
Part~B is intended as an alternative nonlinear formulation rather than as a different approximation space.  For the exact finite plane-wave recovery problems, Parts~A and B therefore give the same field once the correct directions have been found.  We use the Hankel field only as a nonzero-residual check of the variable-projection implementation.  At $\kappa=8$ on the eight-element mesh, one real direction per element initialized from the centroid-radial field gives $\mathcal J=4.5604\times10^{-1}$ and $E_{L^2}=3.5462\times10^{-1}$; a blind Sobol initialization converges to a poorer basin with $\mathcal J=1.5792$ and $E_{L^2}=9.8983\times10^{-1}$.  With four local rays the same variable-projection calculation gives $\mathcal J=2.5410\times10^{-2}$ and $E_{L^2}=6.9700\times10^{-2}$.  These checks confirm the implementation and the same nonconvexity already seen in Part~A; they are not used as evidence for a separate approximation advantage.

\section{Conclusion}\label{sec:conclusion}
In this paper we have considered PWDG approximations in which the propagation directions are determined during the computation.  In Part~A the PWDG equations are retained and the skeleton residual is minimized over the directions.  In Part~B the Galerkin constraint is removed and the linear coefficients are eliminated by variable projection.  Complex angles allow the same local family to contain propagating and evanescent waves.

The numerical experiments show that the usefulness of direction adaptation depends strongly on the directional content of the solution.  For the finite plane-wave sums, the result is particularly clear.  In the nested family considered here, ENRICH--MOVE continuation recovers all directions to roundoff for $M=1,\ldots,19$, while a uniform angular space with the same number of coefficients remains much less accurate.  At $M=20$ the automatic birth step enters a false basin.  Repeating the same step with a nearby initial direction, still in double precision, again gives roundoff error.  Thus this failure is caused by the nonlinear identification of the new direction and not by the twenty-ray approximation space.  We do not interpret nineteen as a universal limit; the experiment is intended to show the regime in which direction adaptation gives its largest gain.

The Hankel calculation gives the complementary result.  On the exact circular DtN boundary the learned fan follows the local phase direction accurately, including for an off-center source, but the field error remains much larger.  A small number of plane waves cannot exactly represent the curvature and amplitude variation of a cylindrical field.  Hence learning the correct local phase direction is not by itself sufficient when the directional content is broad.

We also revisited the high-$p$ numerical instability observed in the earlier DtN--PWDG calculations.  The experiments separate three effects.  The raw coefficient basis loses accuracy first.  Graph--Riesz normalization improves the calculation on the retained trace space.  Decreasing $\tau_{\rm rank}$ then restores additional trace modes.  Finally, when the smallest trace eigenvalues are computed at 50 decimal digits, the modes that appeared nonpositive in double precision are positive.  This shows that $\tau_{\rm rank}$ should be regarded as a numerical parameter depending on the working precision and required accuracy, rather than as a fixed constant.

For straight edges, the products of the local plane waves and their direction derivatives reduce to exact exponential moments.  In the all-Dirichlet case we also have
\[
 \mathcal J(v_h)=\|u-v_h\|_{\mathcal G,h}^2,
\]
so that the residual used to move the directions is the same quantity used to measure the graph error.  For mixed Dirichlet--impedance problems we obtained the corresponding $L^2$ residual bound under the broken Trefftz stability assumption of Section~\ref{subsec:mixed-control}.

The transmission experiment shows that the complex-angle formulation recovers both the propagating Snell direction and an evanescent transmitted wave.  In the combined adaptive experiment, the usual SOLVE--ESTIMATE--MARK loop is supplemented by a TEST step which chooses between MOVE, ENRICH and REFINE on marked elements.  The relative graph error falls below $5\%$ with 3329 Trefftz coefficients, compared with 11290 coefficients for pure $h$-refinement.  This is a comparison of approximation spaces; the present TEST step uses several global trial solves and is not yet a cheaper algorithm in elapsed time.

The nonlinear analysis given here is local.  On a fixed-rank neighborhood of an identifiable zero-residual solution we obtain quadratic growth of the reduced objective and stability of the field with respect to the direction parameters.  This agrees with the numerical behavior: once the iteration is in the correct basin the directions converge rapidly, while a poor birth or a direct solve at the final frequency may converge to another minimum.  An important next step is therefore to replace the present global trial solves by cheaper local tests and to improve the continuation and birth strategies when several directional components are present.

\section*{Declarations}
\textbf{Competing interests.} The author declares no competing interests.

\textbf{Data and code availability.} The numerical experiments are generated by the accompanying research code and tabulated data supplied with the manuscript.  The computational results reported in the paper are not extrapolated from fitted curves.

\textbf{Author contributions.} S. Kapita developed the methodology, analysis, implementation, numerical experiments, and manuscript.


\begin{thebibliography}{99}
\bibitem{CessenatDespres1998}
O.~Cessenat and B.~Despr\'es.
\newblock Application of an ultra weak variational formulation of elliptic PDEs to the two-dimensional Helmholtz problem.
\newblock \emph{SIAM J. Numer. Anal.}, 35(1):255--299, 1998.

\bibitem{MonkWang1999}
P.~Monk and D.-Q.~Wang.
\newblock A least-squares method for the Helmholtz equation.
\newblock \emph{Comput. Methods Appl. Mech. Engrg.}, 175:121--136, 1999.

\bibitem{GolubPereyra1973}
G.~H.~Golub and V.~Pereyra.
\newblock The differentiation of pseudo-inverses and nonlinear least squares problems whose variables separate.
\newblock \emph{SIAM J. Numer. Anal.}, 10(2):413--432, 1973.

\bibitem{GittelsonHiptmairPerugia2009}
C.~J. Gittelson, R.~Hiptmair, and I.~Perugia.
\newblock Plane wave discontinuous Galerkin methods: analysis of the $h$-version.
\newblock \emph{ESAIM Math. Model. Numer. Anal.}, 43(2):297--331, 2009.

\bibitem{HiptmairMoiolaPerugia2011}
R.~Hiptmair, A.~Moiola, and I.~Perugia.
\newblock Plane wave discontinuous Galerkin methods for the 2D Helmholtz equation: analysis of the $p$-version.
\newblock \emph{SIAM J. Numer. Anal.}, 49(1):264--284, 2011.

\bibitem{HiptmairMoiolaPerugia2016}
R.~Hiptmair, A.~Moiola, and I.~Perugia.
\newblock Plane wave discontinuous Galerkin methods: exponential convergence of the $hp$-version.
\newblock \emph{Found. Comput. Math.}, 16(3):637--675, 2016.

\bibitem{HiptmairMoiolaPerugiaSurvey2016}
R.~Hiptmair, A.~Moiola, and I.~Perugia.
\newblock A survey of Trefftz methods for the Helmholtz equation.
\newblock In \emph{Building Bridges: Connections and Challenges in Modern Approaches to Numerical Partial Differential Equations}, Lecture Notes in Computational Science and Engineering 114, pp.~237--279. Springer, Cham, 2016.

\bibitem{KapitaMonk2018}
S.~Kapita and P.~Monk.
\newblock A plane wave discontinuous Galerkin method with a Dirichlet-to-Neumann boundary condition for the scattering problem in acoustics.
\newblock \emph{J. Comput. Appl. Math.}, 327:208--225, 2018.

\bibitem{KapitaMonkWarburton2015}
S.~Kapita, P.~Monk, and T.~Warburton.
\newblock Residual-based adaptivity and PWDG methods for the Helmholtz equation.
\newblock \emph{SIAM J. Sci. Comput.}, 37(3):A1525--A1553, 2015.

\bibitem{CongreveHoustonPerugia2019}
S.~Congreve, P.~Houston, and I.~Perugia.
\newblock Adaptive refinement for $hp$-version Trefftz discontinuous Galerkin methods for the homogeneous Helmholtz problem.
\newblock \emph{Adv. Comput. Math.}, 45(1):361--393, 2019.

\bibitem{CongreveGedickePerugia2019}
S.~Congreve, J.~Gedicke, and I.~Perugia.
\newblock Numerical investigation of the conditioning for plane wave discontinuous Galerkin methods.
\newblock In F.~A.~Radu, K.~Kumar, I.~Berre, J.~M.~Nordbotten, and I.~S.~Pop (eds.), \emph{Numerical Mathematics and Advanced Applications -- ENUMATH 2017}, Lecture Notes in Computational Science and Engineering 126, pp.~493--500. Springer, Cham, 2019.

\bibitem{BarucqBendaliDiazTordeux2021}
H.~Barucq, A.~Bendali, J.~Diaz, and S.~Tordeux.
\newblock Local strategies for improving the conditioning of the plane-wave ultra-weak variational formulation.
\newblock \emph{J. Comput. Phys.}, 441:110449, 2021.

\bibitem{AgrawalHoppe2017}
A.~Agrawal and R.~H.~W. Hoppe.
\newblock Optimization of plane wave directions in plane wave discontinuous Galerkin methods for the Helmholtz equation.
\newblock \emph{Port. Math.}, 74(1):69--89, 2017.

\bibitem{AmaraChaudhryDiazDjellouliFiedler2014}
M.~Amara, S.~Chaudhry, J.~Diaz, R.~Djellouli, and S.~L. Fiedler.
\newblock A local wave tracking strategy for efficiently solving mid- and high-frequency Helmholtz problems.
\newblock \emph{Comput. Methods Appl. Mech. Engrg.}, 276:473--508, 2014.

\bibitem{FangQianZepedaZhao2017}
J.~Fang, J.~Qian, L.~Zepeda-N\'u\~nez, and H.~Zhao.
\newblock Learning dominant wave directions for plane wave methods for high-frequency Helmholtz equations.
\newblock \emph{Res. Math. Sci.}, 4:9, 2017.

\bibitem{LamQian2019}
C.~Y.~Lam and J.~Qian.
\newblock Numerical microlocal analysis by fast Gaussian wave packet transforms and application to high-frequency Helmholtz problems.
\newblock \emph{SIAM J. Sci. Comput.}, 41(5):A2717--A2746, 2019.

\bibitem{HuWang2021}
Q.~Hu and Z.~Wang.
\newblock A plane wave method based on approximate wave directions for two dimensional Helmholtz equations with large wave numbers.
\newblock arXiv:2107.09797, 2021.

\bibitem{ParolinHuybrechsMoiola2023}
E.~Parolin, D.~Huybrechs, and A.~Moiola.
\newblock Stable approximation of Helmholtz solutions in the disk by evanescent plane waves.
\newblock \emph{ESAIM Math. Model. Numer. Anal.}, 57(6):3499--3536, 2023.

\bibitem{BetckeTrefethen2005}
T.~Betcke and L.~N.~Trefethen.
\newblock Reviving the method of particular solutions.
\newblock \emph{SIAM Rev.}, 47(3):469--491, 2005.

\bibitem{BarnettBetcke2008}
A.~H.~Barnett and T.~Betcke.
\newblock Stability and convergence of the method of fundamental solutions for Helmholtz problems on analytic domains.
\newblock \emph{J. Comput. Phys.}, 227(14):7003--7026, 2008.

\bibitem{MoiolaSpence2019}
A.~Moiola and E.~A.~Spence.
\newblock Acoustic transmission problems: wavenumber-explicit bounds and resonance-free regions.
\newblock \emph{Math. Models Methods Appl. Sci.}, 29(2):317--354, 2019.

\bibitem{BaskinSpenceWunsch2016}
D.~Baskin, E.~A.~Spence, and J.~Wunsch.
\newblock Sharp high-frequency estimates for the Helmholtz equation and applications to boundary integral equations.
\newblock \emph{SIAM J. Math. Anal.}, 48(1):229--267, 2016.

\bibitem{Dorfler1996}
W.~D\"orfler.
\newblock A convergent adaptive algorithm for Poisson's equation.
\newblock \emph{SIAM J. Numer. Anal.}, 33(3):1106--1124, 1996.

\bibitem{HjorungnesGesbert2007}
A.~Hj\o rungnes and D.~Gesbert.
\newblock Complex-valued matrix differentiation: techniques and key results.
\newblock \emph{IEEE Trans. Signal Process.}, 55(6):2740--2746, 2007.

\bibitem{More1978}
J.~J.~Mor\'e.
\newblock The Levenberg--Marquardt algorithm: implementation and theory.
\newblock In G.~A.~Watson (ed.), \emph{Numerical Analysis}, Lecture Notes in Mathematics 630, pp.~105--116. Springer, Berlin, 1978.

\bibitem{HestenesStiefel1952}
M.~R.~Hestenes and E.~Stiefel.
\newblock Methods of conjugate gradients for solving linear systems.
\newblock \emph{J. Res. Natl. Bur. Stand.}, 49(6):409--436, 1952.

\bibitem{SaadSchultz1986}
Y.~Saad and M.~H.~Schultz.
\newblock GMRES: a generalized minimal residual algorithm for solving nonsymmetric linear systems.
\newblock \emph{SIAM J. Sci. Stat. Comput.}, 7(3):856--869, 1986.

\bibitem{VanderVorst1992}
H.~A.~van der Vorst.
\newblock Bi-CGSTAB: a fast and smoothly converging variant of Bi-CG for the solution of nonsymmetric linear systems.
\newblock \emph{SIAM J. Sci. Stat. Comput.}, 13(2):631--644, 1992.

\bibitem{Sobol1967}
I.~M.~Sobol'.
\newblock On the distribution of points in a cube and the approximate evaluation of integrals.
\newblock \emph{USSR Comput. Math. Math. Phys.}, 7(4):86--112, 1967.

\bibitem{Watson1989}
L.~T.~Watson.
\newblock Globally convergent homotopy methods: a tutorial.
\newblock \emph{Appl. Math. Comput.}, 31:369--396, 1989.

\bibitem{MardalWinther2011}
K.-A.~Mardal and R.~Winther.
\newblock Preconditioning discretizations of systems of partial differential equations.
\newblock \emph{Numer. Linear Algebra Appl.}, 18(1):1--40, 2011.
\end{thebibliography}
\end{document}